\documentclass[a4paper,twoside]{article}
\usepackage{a4}
\usepackage{amssymb}
\usepackage{amsmath}
\usepackage{upref}
\usepackage[active]{srcltx}
\usepackage[colorlinks,citecolor=blue,linkcolor=blue]{hyperref}
\allowdisplaybreaks[2] 
\newcount\minutes \newcount\hours
\hours=\time
\divide\hours 60
\minutes=\hours
\multiply\minutes -60
\advance\minutes \time
\newcommand{\klockan}{\the\hours:{\ifnum\minutes<10 0\fi}\the\minutes}
\newcommand{\tid}{\today\ \klockan}
\newcommand{\prtid}{\smash{\raise 10mm \hbox{\LaTeX ed \tid}}}
\renewcommand{\prtid}{}
\makeatletter
\def\sectionmark#1{} 
\def\subsectionmark#1{}
\newcommand{\sectnr}{\ifnum \c@secnumdepth >\z@
                 \thesection.\hskip 1em\relax \fi}
\def\@evenhead{\footnotesize\rm\thepage\hfil\leftmark\hfil\llap{\prtid}}
\def\@oddhead{\footnotesize\rm\rlap{\prtid}\hfil\rightmark\hfil\thepage}
\def\tableofcontents{\section*{Contents} 
 \@starttoc{toc}}
\makeatother
\makeatletter
\def\@biblabel#1{#1.}
\makeatother
\makeatletter
\let\Thebibliography=\thebibliography
\renewcommand{\thebibliography}[1]{\def\@mkboth##1##2{}\Thebibliography{#1}
\addcontentsline{toc}{section}{References}
\frenchspacing 
\setlength{\@topsep}{0pt}
\setlength{\itemsep}{0pt}%
\setlength{\parskip}{0pt plus 2pt}%
}
\makeatother
\makeatletter
\def\mdots@{\mathinner.\nonscript\!.%
 \ifx\next,.\else\ifx\next;.\else\ifx\next..\else
 \nonscript\!\mathinner.\fi\fi\fi}
\let\ldots\mdots@
\let\cdots\mdots@
\let\dotso\mdots@
\let\dotsb\mdots@
\let\dotsm\mdots@
\let\dotsc\mdots@
\def\vdots{\vbox{\baselineskip2.8\p@ \lineskiplimit\z@
    \kern6\p@\hbox{.}\hbox{.}\hbox{.}\kern3\p@}}
\def\ddots{\mathinner{\mkern1mu\raise8.6\p@\vbox{\kern7\p@\hbox{.}}%
    \raise5.8\p@\hbox{.}\raise3\p@\hbox{.}\mkern1mu}}
\makeatother
\makeatletter
\let\Enumerate=\enumerate
\renewcommand{\enumerate}{\Enumerate%
\setlength{\@topsep}{0pt}
\setlength{\itemsep}{0pt}%
\setlength{\parskip}{0pt plus 1pt}%
\renewcommand{\theenumi}{\textup{(\alph{enumi})}}%
\renewcommand{\labelenumi}{\theenumi}%
}
\let\endEnumerate=\endenumerate
\renewcommand{\endenumerate}{\endEnumerate\unskip}
\makeatother
\makeatletter
\def\@seccntformat#1{\csname the#1\endcsname.\quad}
\makeatother
\newcommand{\authortitle}[2]{\author{#1}\title{#2}\markboth{#1}{#2}}
\newcommand{\auth}[2]{{#1, #2.}}
\newcommand{\art}[6]{{\sc #1, \rm #2, \it #3\/ \bf #4 \rm (#5), \mbox{#6}.}}
\newcommand{\artprep}[3]{{\sc #1, \rm #2, \rm #3.}}
\newcommand{\arttoappear}[3]{{\sc #1, \rm #2, to appear in \it #3}}
\newcommand{\book}[3]{{\sc #1, \it #2, \rm #3.}}
\newcommand{\AND}{{\rm and }}
\RequirePackage{amsthm}
\newtheoremstyle{descriptive}%
  {\topsep}   
  {\topsep}   
  {\rmfamily} 
  {}          
  {\bfseries} 
  {.}         
  { }         
  {}          
\newtheoremstyle{propositional}%
  {\topsep}   
  {\topsep}   
  {\itshape}  
  {}          
  {\bfseries} 
  {.}         
  { }         
  {}          
\theoremstyle{propositional}
\newtheorem{thm}{Theorem}[section]
\newtheorem{prop}[thm]{Proposition}
\newtheorem{lem}[thm]{Lemma}
\newtheorem{cor}[thm]{Corollary}
\theoremstyle{descriptive}
\newtheorem{deff}[thm]{Definition}
\newtheorem{example}[thm]{Example}
\newtheorem{remark}[thm]{Remark}
\newtheorem{openprob}[thm]{Open problem}
\makeatletter
\renewenvironment{proof}[1][\proofname]{\par
  \pushQED{\qed}%
  \normalfont
  \trivlist
  \item[\hskip\labelsep
        \itshape
    #1\@addpunct{.}]\ignorespaces
}{%
  \popQED\endtrivlist\@endpefalse
}
\makeatother
\newcommand{\setm}{\setminus}
\renewcommand{\emptyset}{\varnothing}
\renewcommand{\subsetneq}{\varsubsetneq}
\renewcommand{\varsubsetneqq}{\varsubsetneq}
\def\vint{\mathop{\mathchoice%
          {\setbox0\hbox{$\displaystyle\intop$}\kern 0.22\wd0%
           \vcenter{\hrule width 0.6\wd0}\kern -0.82\wd0}%
          {\setbox0\hbox{$\textstyle\intop$}\kern 0.2\wd0%
           \vcenter{\hrule width 0.6\wd0}\kern -0.8\wd0}%
          {\setbox0\hbox{$\scriptstyle\intop$}\kern 0.2\wd0%
           \vcenter{\hrule width 0.6\wd0}\kern -0.8\wd0}%
          {\setbox0\hbox{$\scriptscriptstyle\intop$}\kern 0.2\wd0%
           \vcenter{\hrule width 0.6\wd0}\kern -0.8\wd0}}%
          \mathopen{}\int}
\newcommand{\Cp}{C_p}
\newcommand{\CpX}{{C_p^X}}
\newcommand{\CpXhat}{{C_p^{\Xhat}}}
\newcommand{\CpOm}{{C_p^\Om}}
\newcommand{\grad}{\nabla}
\DeclareMathOperator{\Lip}{Lip}
\DeclareMathOperator{\dist}{dist}
\DeclareMathOperator{\diam}{diam}
\DeclareMathOperator{\dvg}{div}
\DeclareMathOperator{\supp}{supp}
\DeclareMathOperator*{\essliminf}{ess\,lim\,inf}
\newcommand{\bdry}{\partial}
\newcommand{\loc}{_{\rm loc}}
{\catcode`p =12 \catcode`t =12 \gdef\eeaa#1pt{#1}}      
\def\accentadjtext#1{\setbox0\hbox{$#1$}\kern   
                \expandafter\eeaa\the\fontdimen1\textfont1 \ht0 }
\def\accentadjscript#1{\setbox0\hbox{$#1$}\kern 
                \expandafter\eeaa\the\fontdimen1\scriptfont1 \ht0 }
\def\accentadjscriptscript#1{\setbox0\hbox{$#1$}\kern   
                \expandafter\eeaa\the\fontdimen1\scriptscriptfont1 \ht0 }
\def\accentadjtextback#1{\setbox0\hbox{$#1$}\kern       
                -\expandafter\eeaa\the\fontdimen1\textfont1 \ht0 }
\def\accentadjscriptback#1{\setbox0\hbox{$#1$}\kern     
                -\expandafter\eeaa\the\fontdimen1\scriptfont1 \ht0 }
\def\accentadjscriptscriptback#1{\setbox0\hbox{$#1$}\kern 
                -\expandafter\eeaa\the\fontdimen1\scriptscriptfont1 \ht0 }
\def\itoverline#1{{\mathsurround0pt\mathchoice
        {\rlap{$\accentadjtext{\displaystyle #1}
                \accentadjtext{\vrule height1.593pt}
                \overline{\phantom{\displaystyle #1}
                \accentadjtextback{\displaystyle #1}}$}{#1}}
        {\rlap{$\accentadjtext{\textstyle #1}
                \accentadjtext{\vrule height1.593pt}
                \overline{\phantom{\textstyle #1}
                \accentadjtextback{\textstyle #1}}$}{#1}}
        {\rlap{$\accentadjscript{\scriptstyle #1}
                \accentadjscript{\vrule height1.593pt}
                \overline{\phantom{\scriptstyle #1}
                \accentadjscriptback{\scriptstyle #1}}$}{#1}}
        {\rlap{$\accentadjscriptscript{\scriptscriptstyle #1}
                \accentadjscriptscript{\vrule height1.593pt}
                \overline{\phantom{\scriptscriptstyle #1}
                \accentadjscriptscriptback{\scriptscriptstyle #1}}$}{#1}}}}
\newcommand{\al}{\alpha}
\newcommand{\ga}{\gamma}
\newcommand{\dmu}{d\mu}
\newcommand{\de}{\delta}
\newcommand{\eps}{\varepsilon}
\newcommand{\la}{\lambda}
\newcommand{\Om}{\Omega}
\renewcommand{\phi}{\varphi}
\newcommand{\p}{{$p\mspace{1mu}$}}
\newcommand{\R}{\mathbf{R}}
\newcommand{\eR}{{\overline{\R}}}
\newcommand{\Q}{\mathbf{Q}}
\newcommand{\limplus}{{\mathchoice{\raise.17ex\hbox{$\scriptstyle +$}}
                {\raise.17ex\hbox{$\scriptstyle +$}}
                {\raise.1ex\hbox{$\scriptscriptstyle +$}}
                {\scriptscriptstyle +}}}
\newcommand{\limminus}{{\mathchoice{\raise.17ex\hbox{$\scriptstyle -$}}
                {\raise.17ex\hbox{$\scriptstyle -$}}
                {\raise.1ex\hbox{$\scriptscriptstyle -$}}
                {\scriptscriptstyle -}}}
\newcommand{\Np}{N^{1,p}}
\newcommand{\Nploc}{N^{1,p}\loc}
\newcommand{\Nxloc}{\Np_{\rm loc,{\bf x}}}
\newcommand{\Ncptloc}{\Np_{\rm loc,cpt}}
\newcommand{\Ndistloc}{\Np_{\rm loc,dist}}
\newcommand{\Nbdyloc}{\Np_{\rm loc,bdy}}
\newcommand{\Nclosloc}{\Np_{\rm loc,clos}}
\newcommand{\Nxo}{\Np_{\rm 0,{\bf x}}}
\newcommand{\Ncpto}{\Np_{\rm 0,cpt}}
\newcommand{\Ndisto}{\Np_{\rm 0,dist}}
\newcommand{\Nbdyo}{\Np_{\rm 0,bdy}}
\newcommand{\Ncloso}{\Np_{\rm 0,clos}}
\newcommand{\Nyo}{\Np_{\rm 0,{\bf y}}}

\newcommand{\Dp}{D^p}
\newcommand{\Dploc}{D^{p}\loc}
\newcommand{\Lploc}{L^p\loc}
\newcommand{\Ga}{\Gamma}
\newcommand{\Xhat}{{\widehat{X}}}
\newcommand{\xhat}{{\hat{x}}}
\newcommand{\Omhat}{\Om^\wedge}
\newcommand{\Bhat}{{\widehat{B}}}
\newcommand{\Ghat}{{\widehat{G}}}
\newcommand{\ghat}{{\hat{g}}}
\newcommand{\uhat}{{\hat{u}}}
\newcommand{\phihat}{{\widehat{\phi}}}
\newcommand{\ut}{{\tilde{u}}}
\newcommand{\clG}{\itoverline{G}}
\newcommand{\Bik}{B_{ik}}
\newcommand{\Bhatik}{{B^\wedge_{ik}}}
\newcommand{\rik}{r_{ik}}
\newcommand{\phiik}{\phi_{ik}}
\newcommand{\Bjk}{B_{jk}}
\newcommand{\rjk}{r_{jk}}
\newcommand{\Gx}{{\mathcal G}_{\rm\bf x}}
\newcommand{\Ex}{{\mathcal E}_{\rm\bf x}}
\newcommand{\Gcpt}{{\mathcal G}_{\rm cpt}}

\newcommand{\Gdist}{{\mathcal G}_{\rm dist}}
\newcommand{\Gbdy}{{\mathcal G}_{\rm bdy}}
\newcommand{\Gclos}{{\mathcal G}_{\rm clos}}
\newcommand{\CPI}{C_{\rm PI}}
\newcommand{\Npc}{{N^{1,p}_c}}
\newcommand{\Borel}{\mathcal{B}}
\newcommand{\Borelh}{\widehat{\mathcal{B}}}
\newcommand{\Meas}{\mathcal{M}}
\newcommand{\Meash}{\widehat{\mathcal{M}}}
\numberwithin{equation}{section}
\newenvironment{ack}{\medskip{\it Acknowledgement.}}{}

\begin{document}

\authortitle{Anders Bj\"orn and Jana Bj\"orn}
{Poincar\'e inequalities and Newtonian Sobolev functions 
on noncomplete metric spaces}

\author{
Anders Bj\"orn \\
\it\small Department of Mathematics, Link\"oping University, \\
\it\small SE-581 83 Link\"oping, Sweden\/{\rm ;}
\it \small anders.bjorn@liu.se
\\
\\
Jana Bj\"orn \\
\it\small Department of Mathematics, Link\"oping University, \\
\it\small SE-581 83 Link\"oping, Sweden\/{\rm ;}
\it \small jana.bjorn@liu.se
}

\date{}

\maketitle

\noindent{\small {\bf Abstract}. 
Let $X$ be a noncomplete metric space satisfying the usual (local) 
assumptions of a doubling property and a Poincar\'e inequality.
We study extensions of Newtonian Sobolev functions to the completion 
$\Xhat$ of $X$ and use them to obtain several results on $X$ itself,
in particular concerning minimal weak upper gradients,
Lebesgue points, quasicontinuity, 
regularity properties of the capacity and better Poincar\'e inequalities.
We also provide a discussion about possible applications of the 
completions and extension results to \p-harmonic functions on noncomplete
spaces and show by examples that this is a rather delicate issue opening 
for various interpretations and new investigations.

\bigskip
\noindent
{\small \emph{Key words and phrases}:
Lebesgue point,
local doubling,
locally compact metric space,
noncomplete metric space, 
Newtonian space,
nonlinear potential theory, 
\p-harmonic function,
Poincar\'e inequality,
quasicontinuity,
quasiminimizer,
semilocal doubling,
Sobolev space.
}

}

\section{Introduction}

Our aim in this paper is to study 
Poincar\'e inequalities and Newtonian (Sobolev) functions 
on noncomplete metric spaces,   
and primarily to do so using their completion. 
This turns out to be a rather fruitful approach which, however, has
certain subtleties and limitations, in particular when dealing with
\p-harmonic functions.

Let $X=(X,d,\mu)$ be a metric measure spaces,
where $\mu$ is a positive complete  Borel  measure $\mu$
such that $0<\mu(B)<\infty$ for all balls $B \subset X$.
We let $\Xhat$ be the completion of $X$ with respect to the metric $d$,
and extend $d$ and $\mu$ to $\Xhat$ so that $\mu(\Xhat \setm X)=0$.
Also let $1<p<\infty$.

Much of analysis on metric spaces has been done assuming
global doubling and global Poincar\'e inequalities,
which for instance are assumed in the monographs
Haj\l asz--Koskela~\cite{HaKo}, Bj\"orn--Bj\"orn~\cite{BBbook} and
Heinonen--Koskela--Shanmugalingam--Tyson~\cite{HKSTbook}.
For wider applicability we study properties
that hold under more local assumptions.
Such assumptions have earlier been considered e.g.\
by Cheeger~\cite{Cheeg},
Danielli--Garofalo--Marola~\cite{DaGaMa}, Garofalo--Marola~\cite{GaMa} 
and Holopainen--Shan\-mu\-ga\-lin\-gam~\cite{HoSh}. 
In the following definition we
follow the recent terminology from Bj\"orn--Bj\"orn~\cite{BBsemilocal},
where a more extensive discussion of these assumptions can be found.

\begin{deff} \label{def-local-intro}
The measure \emph{$\mu$ is doubling within a ball $B(x_0,r_0)$}
if there is $C>0$ (depending on $x_0$ and $r_0$)
such that $\mu(2B)\le C \mu(B)$ for all balls $B \subset B(x_0,r_0)$.

Similarly, the
\emph{\p-Poincar\'e inequality holds within a ball $B(x_0,r_0)$} 
if there are constants $C>0$ and $\lambda \ge 1$ (depending on $x_0$ and $r_0$)
such that for all balls $B\subset B(x_0,r_0)$, 
all integrable functions $u$ on $\la B$, and all upper gradients $g$ of $u$, 
\[ 
        \vint_{B} |u-u_B| \,\dmu
        \le C r_B \biggl( \vint_{\lambda B} g^{p} \,\dmu \biggr)^{1/p},
\] 
where $ u_B 
 :=\vint_B u \,d\mu 
:= \int_B u\, d\mu/\mu(B)$.
These properties are called \emph{local} if 
for every $x_0 \in X$ there is $r_0>0$ 
(depending on $x_0$) such that the doubling property or the \p-Poincar\'e
inequality holds within $B(x_0,r_0)$.
They are called \emph{semilocal} if they hold within every ball in $X$.
\end{deff}

Our first observation is 
that if $\mu$ is doubling (resp.\ supports a \p-Poincar\'e inequality)
within a ball $B(x_0,r_0)$ in $X$  then its zero extension
is also doubling 
(resp.\ supports a \p-Poincar\'e inequality) within the
corresponding ball $\Bhat(x_0,r_0)$ in $\Xhat$.
In particular, this means that the  
semilocal assumptions
extend from $X$ to $\Xhat$, see Corollaries~\ref{cor-semilocal-doubling}
and~\ref{cor-Xhat-PI}.
On the other hand, local doubling (resp.\ a local \p-Poincar\'e inequality)
on $X$ does not extend to $\Xhat$,
even though it does extend to a locally compact open subset of $\Xhat$
containing $X$ (see Lemma~\ref{lem-ext-Nploc}), which may be sufficient
for many applications.

The following extension result is one of the main results in this paper.
(See Theorem~\ref{thm-Xhat-semi} for a more extensive version.)
For an open set $\Om$ in $X$, we let 
\[
\Omhat=\Xhat \setm \itoverline{X \setm \Om},
\]
where the closure is taken in $\Xhat$. 
This makes $\Omhat$ into 
the largest open set in $\Xhat$ such that
$\Om = \Omhat \cap X$.

\begin{thm} \label{thm-Xhat-intro}
Assume that the doubling property
and the \p-Poincar\'e inequality hold within the ball $B_0$
in the sense of Definition~\ref{def-local-intro}. 
Let $\Om\subset B_0$ be open and $u \in \Np(\Om)$.
Then the function
\[
    \uhat(x)=\limsup_{r \to 0} \vint_{\Bhat(x,r) \cap X} u \, d\mu,
   \quad x \in \Omhat,
\]
belongs to $\Np(\Omhat)$ and is a pointwise extension of a representative
of $u$ to $\Omhat$.
Moreover, the minimal \p-weak upper gradients
$g_{\uhat}$ and $g_u$ of $\uhat$ and $u$ with respect to $\Xhat$ and $X$
satisfy
\[ 
  g_{\uhat} \le A_{0} g_u \quad \text{a.e.\ in }\Om,
\] 
where $A_{0}$ is a constant only depending on $p$,
the doubling 
constant and both constants in the \p-Poincar\'e inequality within $B_0$.
\end{thm}

For $\Om=X$, with $X$ locally compact and under global assumptions,
similar extension results appear in 
Aikawa--Shanmugalingam~\cite[Proposition~7.1]{AikSh05} and
Hei\-no\-nen--Kos\-ke\-la--Shan\-mu\-ga\-lin\-gam--Ty\-son~\cite[Lemma~8.2.3]{HKSTbook}.
Theorem~\ref{thm-Xhat-intro}
makes it possible to study functions $u$ on $X$ using properties
known to hold for their extensions $\uhat$ on $\Xhat$.
We use this to obtain 
some $L^p$-Lebesgue point  and quasicontinuity results 
for Newtonian Sobolev functions in noncomplete spaces.

When $X$ is complete and $\mu$ is globally doubling, 
a deep result due to Keith--Zhong~\cite[Theorem~1.0.1]{KZ}
shows that the  Poincar\'e inequality is an open-ended property,
in the sense that if $\mu$  supports a global \p-Poincar\'e inequality then
it also supports a global $q$-Poincar\'e inequality for some $q<p$.
Counterexamples due to Koskela~\cite{Koskela} show that this
is false for  locally compact $X$.
Nevertheless, by localizing the arguments in~\cite{KZ}
local versions of this self-improvement result were obtained in 
\cite{BBsemilocal} for locally compact spaces.
In Section~\ref{sect-self-imp-PI},
we further generalize these results 
to non-locally compact spaces,
using our extension theorem as the key tool.

We end the paper with a discussion on \p-harmonic functions (and more generally
quasiminimizers and quasisuperminimizers) on noncomplete spaces with particular
emphasis on locally compact spaces.
It turns out that the choice of the test function space and the 
local Newtonian space for \p-harmonic functions plays an important 
role for the validity of several of the fundamental properties of \p-harmonic
functions, such as various Harnack inequalities and maximum principles.
There are several different natural choices of these spaces,  which all coincide
in the complete case.

Thus, it is the intended applications and particular results, which 
essentially determine
the ``right definition'' of \p-harmonic functions for various
purposes in noncomplete spaces.
The continuity of \p-harmonic functions is, 
however, possible 
to obtain under most of these definitions, see 
Theorem~\ref{thm-semiloc-int-reg}.

Some of this versatility is demonstrated in Example~\ref{ex-slit-plane} 
and it is for instance
possible to treat mixed and Neumann boundary data
as special cases of Dirichlet data.
In complete spaces, most of the suggested definitions reduce to the usual 
definition of \p-harmonic functions.

\begin{ack}
The authors were supported by the Swedish Research Council, 
grants 621-2014-3974 and 2016-03424.
We thank Nageswari Shanmugalingam for helpful discussions 
on some results in the paper.
\end{ack}

\section{Upper gradients and Newtonian spaces}
\label{sect-prelim}

We assume throughout the paper
that $X=(X,d,\mu)$ is a metric space equipped
with a metric $d$ and a positive complete  Borel  measure $\mu$
such that $0<\mu(B)<\infty$ for all 
balls $B \subset X$.
It follows that $X$ is  separable and Lindel\"of.
We also assume that $1<p<\infty$, although the
results in Sections~\ref{sect-prelim}
and~\ref{sect-doubling} also hold if $p=1$.
Proofs of the results in 
this section can be found in the monographs
Bj\"orn--Bj\"orn~\cite{BBbook} and
Heinonen--Koskela--Shanmugalingam--Tyson~\cite{HKSTbook}.

A \emph{curve} is a continuous mapping from an interval,
and a \emph{rectifiable} curve is a curve with finite length.
Unless said otherwise, we
will only consider curves which are nonconstant, compact and 
rectifiable, and thus each curve can 
be parameterized by its arc length $ds$. 
A property is said to hold for \emph{\p-almost every curve}
if it fails only for a curve family $\Ga$ with zero \p-modulus, 
i.e.\ there exists $0\le\rho\in L^p(X)$ such that 
$\int_\ga \rho\,ds=\infty$ for every curve $\ga\in\Ga$.

Following Heinonen--Kos\-ke\-la~\cite{HeKo98},
we introduce upper gradients 
as follows 
(they called them very weak gradients).

\begin{deff} \label{deff-ug}
A Borel function $g: X \to [0,\infty]$  is an \emph{upper gradient} 
of a function $f:X \to \eR:=[-\infty,\infty]$
if for all  curves  
$\gamma \colon [0,l_{\gamma}] \to X$,
\begin{equation} \label{ug-cond}
        |f(\gamma(0)) - f(\gamma(l_{\gamma}))| \le \int_{\gamma} g\,ds,
\end{equation}
where the left-hand side is considered to be $\infty$ 
whenever at least one of the 
terms therein is infinite.
If $g:X \to [0,\infty]$ is measurable 
and \eqref{ug-cond} holds for \p-almost every curve,
then $g$ is a \emph{\p-weak upper gradient} of~$f$. 
\end{deff}

The \p-weak upper gradients were introduced in
Koskela--MacManus~\cite{KoMc}. 
It was also shown therein
that if $g \in \Lploc(X)$ is a \p-weak upper gradient of $f$,
then one can find a sequence $\{g_j\}_{j=1}^\infty$
of upper gradients of $f$ such that $\|g_j-g\|_{L^p(X)} \to 0$.
If $f$ has an upper gradient in $\Lploc(X)$, then
it has an a.e.\ unique \emph{minimal \p-weak upper gradient} $g_f \in \Lploc(X)$
in the sense that for every \p-weak upper gradient $g \in \Lploc(X)$ of $f$ we have
$g_f \le g$ a.e., see Shan\-mu\-ga\-lin\-gam~\cite{Sh-harm}.
Following Shanmugalingam~\cite{Sh-rev}, 
we define a version of Sobolev spaces on the metric space $X$.

\begin{deff} \label{deff-Np}
For a measurable function $f:X\to \eR$, let 
\[
        \|f\|_{\Np(X)} = \biggl( \int_X |f|^p \, d\mu 
                + \inf_g  \int_X g^p \, d\mu \biggr)^{1/p},
\]
where the infimum is taken over all upper gradients $g$ of $f$.
The \emph{Newtonian space} on $X$ is 
\[
        \Np (X) = \{f: \|f\|_{\Np(X)} <\infty \}.
\]
\end{deff}

The quotient
space $\Np(X)/{\sim}$, where  $f \sim h$ if and only if $\|f-h\|_{\Np(X)}=0$,
is a Banach space and a lattice, see Shan\-mu\-ga\-lin\-gam~\cite{Sh-rev}.
We also define
\[
   \Dp(X)=\{f : f \text{ is measurable and  has an upper gradient
     in }   L^p(X)\}.
\]
In this paper we assume that functions in $\Np(X)$
and $\Dp(X)$
 are defined everywhere (with values in $\eR$),
not just up to an equivalence class in the corresponding function space.
This is important for upper gradients to make sense.

For a measurable set $E\subset X$, the Newtonian space $\Np(E)$ is defined by
considering $(E,d|_E,\mu|_E)$ as a metric space in its own right.
We say  that $f \in \Nploc(E)$ if
for every $x \in E$ there exists a ball $B_x\ni x$ such that
$f \in \Np(B_x \cap E)$.
The spaces $\Dp(E)$ and $\Dploc(E)$ are defined similarly.
If $f,h \in \Dploc(X)$, then $g_f=g_h$ a.e.\ in $\{x \in X : f(x)=h(x)\}$,
in particular for $c \in \R$ we have
$g_{\min\{f,c\}}=g_f \chi_{\{f < c\}}$ a.e.

\begin{deff}
The (Sobolev) \emph{capacity}  of a set $E\subset X$  is the number 
\begin{equation*} 
   \CpX(E)=\Cp(E) =\inf_u    \|u\|_{\Np(X)}^p,
\end{equation*}
where the infimum is taken over all $u\in \Np (X)$ such that $u=1$ on $E$.
\end{deff}

We say that a property holds \emph{quasieverywhere} (q.e.)\ 
if the set of points  for which the property does not hold 
has capacity zero. 
The capacity is the correct gauge 
for distinguishing between two Newtonian functions. 
Namely, if $u \in \Np(X)$ then $u \sim v$ if and only if $u=v$ q.e.
Moreover, if $u,v \in \Dploc(X)$ and $u= v$ a.e., then $u=v$ q.e.

We let $B=B(x,r)=\{y \in X : d(x,y) < r\}$ denote the ball
with centre $x$ and radius $r$, and let $\la B=B(x,\la r)$.
We assume throughout the paper that balls are open.
In metric spaces it can happen that
balls with different centres and/or radii 
denote the same set. 
We will however make the convention that a ball $B$ comes with
a predetermined centre and radius $r_B$. 
Note that it can happen that  $B(x_0,r_0) \subset B(x_1,r_1)$
even when $r_0 > r_1$.
In disconnected spaces this can happen also when $r_0 > 2 r_1$.
If $X$ is connected, then $B(x_0,r_0) \subset B(x_1,r_1)$ with
$r_0 >2r_1$ is possible only when  $B(x_0,r_0)= B(x_1,r_1)=X$.

\section{Local doubling and Poincar\'e inequalities}
\label{sect-doubling}

Our aim in this paper is to study noncomplete spaces $X$ and 
primarily to do so using their completion $\Xhat$.
The completion is taken with respect to the metric $d$, whose extension
to $\Xhat$ is also denoted $d$.
The measure $\mu$ is extended so that $\mu(\Xhat \setm X)=0$
and so that
\[
  \Meash=\{E \subset \Xhat : E \cap X \in \Meas\}
\]
is the $\sigma$-algebra of measurable sets on $\Xhat$,
where $\Meas$ is the $\sigma$-algebra of measurable sets on $X$.

\begin{lem}
$\mu$ is a complete Borel regular measure on $\Xhat$.
  Moreover, if $\Borel$ and $\Borelh$ are the Borel $\sigma$-algebras
 on $X$ and $\Xhat$,
  respectively, then
  \begin{equation} \label{eq-Borel}
      \Borel = \{E \cap X : E \in \Borelh\}.
  \end{equation}
\end{lem}

\begin{proof}
We start by proving \eqref{eq-Borel}.
As $\Borelh$ is a $\sigma$-algebra it follows directly that
$\Borel':=\{E \cap X : E \in \Borelh\}$ is a $\sigma$-algebra,
and since it contains all open sets on $X$ it must contain $\Borel$.
Conversely, 
$\{E \subset \Xhat : E \cap X \in \Borel\}$
    is a $\sigma$-algebra which contains all
    open subsets of $\Xhat$ and hence $\Borelh$, from which it follows that
    $\Borel' \subset \Borel$.
    Thus \eqref{eq-Borel} holds.

    Since $E \subset \Xhat$ has zero outer measure if and only
    if $E \cap X$ has zero measure, it follows that
    $\mu$ is a complete Borel regular measure on $\Xhat$ with the
    $\sigma$-algebra $\Meash$.
\end{proof}

Recall from the introduction that for
an open set $\Om\subset X$, 
\[
\Omhat=\Xhat \setm \itoverline{X \setm \Om},
\]
with the closure taken in $\Xhat$, is
the largest open set in $\Xhat$ such that $\Om = \Omhat \cap X$.
Note that $X^\wedge=\Xhat$.
We denote balls with respect to $\Xhat$ by $\Bhat$ or 
$\Bhat(x,r)=\{y\in\Xhat: d(x,y)<r\}$,
and balls with respect to $X$ by $B$, as before.
Note that we do not assume
any general connection between $B$ and $\Bhat$,
and in particular they may have different centres and radii.
The inclusion $\Bhat(x,r)\subset B(x,r)^\wedge$ can be strict, 
but the difference of the two sets is always of measure zero.

Much of analysis on metric spaces has been done assuming
global doubling and global Poincar\'e inequalities.
Here, we study properties that hold under (semi)local assumptions.

\begin{deff} \label{def-local-doubl-mu}
We say that $\mu$ is \emph{locally doubling} (on $X$) if 
for every $x_0 \in X$ there is $r_0>0$ 
(depending on $x_0$) such that $\mu$ is doubling within $B(x_0,r_0)$
in the sense of Definition~\ref{def-local-intro}.

If $\mu$ is doubling within every ball $B(x_0,r_0)$ then 
it is \emph{semilocally doubling} (on $X$),
and if moreover the doubling constant within $B(x_0,r_0)$ 
is independent of $x_0$ and $r_0$,
then $\mu$ is \emph{globally doubling} (on $X$).
\end{deff}

See Heinonen~\cite{heinonen} for more on doubling measures.
If $\mu$ is locally doubling on $X$ and $\Om \subset X$ is open,
then $\mu$  is also locally doubling on $\Om$.
A similar restriction property fails
for semilocal and global doubling, see \cite[Example~4.3]{BBsemilocal}. 

\begin{prop} \label{prop-X-Xhat-doubl}
The measure $\mu$ on $X$ is doubling  within 
$B(x_0,r_0)$ in the sense of Definition~\ref{def-local-intro}
if and only if 
its zero extension to 
$\Xhat$ is doubling within 
$\Bhat(x_0,r_0)$, with the same doubling constant $C_0$.
\end{prop}

For a corresponding result with global assumptions see
Aikawa--Shanmugalingam~\cite[Proposition~7.1]{AikSh05} and
Hei\-no\-nen--Kos\-ke\-la--Shan\-mu\-ga\-lin\-gam--Ty\-son~\cite[Lemma~8.2.3]{HKSTbook}.

\begin{proof}
The sufficiency follows directly from the fact that 
$\mu(B(x,r))=\mu(\Bhat(x,r))$ for all $x\in X$ and $r>0$.

For the necessity, let $\Bhat(\xhat,r)\subset\Bhat(x_0,r_0)$
and $0<\eps<\tfrac{1}{2}r$ be arbitrary. 
Find $x_\eps\in X$ such that $d(x_\eps,\xhat)<\eps$.
Then 
\[
\mu(\Bhat(\xhat,2r-3\eps)) \le \mu(B(x_\eps,2(r-\eps))) 
   \le C_0 \mu(B(x_\eps,r-\eps)) \le C_0 \mu(\Bhat(\xhat,r)),
\]
since $B(x_\eps,r-\eps)\subset B(x_0,r_0)$. 
Letting $\eps\to0$ in the left-hand side shows that
$\mu(\Bhat(\xhat,2r)) \le C_{0} \mu(\Bhat(\xhat,r))$.
\end{proof}

\begin{cor}  \label{cor-semilocal-doubling}
The measure $\mu$ is semilocally doubling on $X$ if and only if
it is semilocally doubling on $\Xhat$.
\end{cor}

\begin{deff} \label{def-PI}
Let $1 \le q < \infty$.
We say that the 
\emph{$(q,p)$-Poincar\'e inequality holds within $B(x_0,r_0)$} 
if there are constants $C>0$ and $\lambda \ge 1$ (depending on $x_0$ and $r_0$)
such that for all balls $B\subset B(x_0,r_0)$, 
all integrable functions $u$ on $\la B$, and all upper gradients $g$ of $u$, 
\begin{equation} \label{eq-PI}
        \biggl(\vint_{B} |u-u_B|^q \,\dmu\biggr)^{1/q}
        \le C r_B \biggl( \vint_{\lambda B} g^{p} \,\dmu \biggr)^{1/p}.
\end{equation}
We also say that $X$ (or $\mu$) supports 
a \emph{local $(q,p)$-Poincar\'e inequality} (on $X$) if
for every $x_0 \in X$ there is $r_0$ (depending on $x_0$) 
such that the $(q,p)$-Poincar\'e inequality holds within $B(x_0,r_0)$.

If the $(q,p)$-Poincar\'e inequality holds within every ball $B(x_0,r_0)$
then $X$ supports a \emph{semilocal $(q,p)$-Poincar\'e inequality},
and if moreover $C$ and $\la$ are independent of $x_0$ and $r_0$,
then $X$ supports a \emph{global $(q,p)$-Poincar\'e inequality}.

If $q=1$ we usually just write 
\emph{\p-Poincar\'e inequality}.
\end{deff}

The Poincar\'e inequality \eqref{eq-PI} 
can equivalently be required to hold
  for all measurable $u$ on $\la B$ 
and all \p-weak upper gradients $g$ of $u$,
  where the left-hand side is interpreted as $\infty$ if $u_B$ is not
  defined.
  This follows from the proof of Proposition~4.13 in \cite{BBbook}.
  However, the use of the dominated convergence at the end of that proof should
  perhaps be explained more carefully by replacing the last inequality therein by
   \begin{align*}
     \infty    
     & =\biggl( \vint_{B} |u-u_B|^q \,d\mu \biggr)^{1/q} 
     =\lim_{j \to \infty}\biggl( \vint_{B} \min\{j,|u-u_B|^q\} \,d\mu \biggr)^{1/q} \\
     &= \lim_{j \to \infty}\lim_{k \to \infty} \biggl( \vint_{B}
           \min\{j,|u_k-(u_k)_B|^q\} \,d\mu \biggr)^{1/q} 
       \le C \diam (B) \biggl( \vint_{\lambda B} g^{p} \,d\mu \biggr)^{1/p}.
   \end{align*}
   Alternatively Fatou's lemma can be used.

As in the case of the doubling condition, 
local Poincar\'e inequalities are inherited by open subsets, i.e.\ 
if $\Om \subset X$ is open and
$X$ supports a local $(q,p)$-Poincar\'e inequality,
then so does $\Om$.
This fails for semilocal and global Poincar\'e inequalities,
see \cite[Example~4.3]{BBsemilocal}.

\begin{prop}   \label{prop-X-Xhat-PI}
If $X$ supports a 
$(q,p)$-Poincar\'e inequality within $B(x_0,r_0)$ in the sense of 
Definition~\ref{def-PI}, with constants $\CPI$ and $\la$.
Then $\Xhat$ supports a $(q,p)$-Poincar\'e inequality 
within $\Bhat(x_0,r_0)$, with the same constants.
\end{prop}

\begin{proof}
Let $\Bhat=\Bhat(\xhat,r)\subset\Bhat(x_0,r_0)$
and $0<\eps<\tfrac{1}{2}r$ be arbitrary. 
Let $u$ be integrable on $\la \Bhat$ and let $\ghat$ be
an upper gradient of $u$ with respect to $\Xhat$.
Then $\ghat|_X$ is an upper gradient of $u$ also with respect to $X$.
By the proof of 
Proposition~4.13 in \cite{BBbook}, we can assume that $u$ is bounded.
Find $x_\eps\in X$ such that $d(x_\eps,\xhat)<\eps$ and let 
$B_\eps:=B(x_\eps,r-\eps)$.
Then 
\[
\Bhat(\xhat,r-2\eps)\cap X\subset B_\eps \subset \Bhat
\quad \text{and}  \quad
\Bhat(\xhat,\la(r-2\eps))\cap X\subset \la B_\eps \subset \la\Bhat,
\]
which implies that
\[
\mu(B_\eps)\to\mu(\Bhat), \quad \mu(\la B_\eps)\to\mu(\la\Bhat) 
 \quad \text{and}  \quad
u_{B_\eps}\to u_{\Bhat}, \quad \text{as }\eps\to0.
\]
Since $B_\eps\subset B(x_0,r_0)$, the $(q,p)$-Poincar\'e inequality on $X$
implies that
\begin{align*}
\biggl(\vint_{B_\eps} |u- u_{B_\eps}|^q \, d\mu \biggr)^{1/q}
   & \le \CPI r \biggl(\vint_{\la B_\eps} \ghat|_X^p \,d\mu\biggr)^{1/p} 
    \le \CPI r \biggl(\frac{\mu(\la\Bhat)}{\mu(\la B_\eps)} 
               \vint_{\la \Bhat} \ghat^p \,d\mu\biggr)^{1/p}
\end{align*}
and letting $\eps\to0$ concludes the proof, by dominated convergence.
\end{proof}

\begin{cor} \label{cor-Xhat-PI}
If $X$ supports a semilocal $(q,p)$-Poincar\'e inequality,
then so does~$\Xhat$.
\end{cor}

There is no equivalence in Proposition~\ref{prop-X-Xhat-PI} or
Corollary~\ref{cor-Xhat-PI}, as is easily
seen by considering $X=\R \setm \Q$.
For corresponding results with global assumptions see
Aikawa--Shanmugalingam~\cite[Proposition~7.1]{AikSh05} and
Hei\-no\-nen--Kos\-ke\-la--Shan\-mu\-ga\-lin\-gam--Ty\-son~\cite[Lemma~8.2.3]{HKSTbook}.

Note that, in spite of Propositions~\ref{prop-X-Xhat-doubl}
and~\ref{prop-X-Xhat-PI},
neither local doubling nor local Poincar\'e inequalities extend to $\Xhat$.
Indeed, the Lebesgue measure on any open set $X\subset\R^n$ is locally 
doubling and supports a local 1-Poincar\'e inequality,
whereas for the completion $\Xhat\subset\R^n$ these properties 
hold only in special cases.
(A typical example where the local doubling property fails is the closed
outer cusp of exponential type, while Poincar\'e inequalities usually
fail on disconnected (or essentially disconnected) sets, such as  
the bow-tie in \cite[Example~A.23]{BBbook}.
See also \cite[Example~4.3]{BBsemilocal}.)

In fact, for the completion $\Xhat$, the local and semilocal
properties are essentially equivalent.
Indeed, this follows from the following proposition.

\begin{prop} \label{prop-semilocal-doubling-mu-X}
{\rm (\cite[Proposition~1.2 and Theorem~4.4]{BBsemilocal})}
If $X$ is proper then $\mu$  is locally doubling if and only if it
is semilocally doubling.

If $X$ is, in addition, connected then also the local and 
semilocal $(q,p)$-Poincar\'e inequalities are equivalent.
\end{prop}

The space $X$ is \emph{proper} if  all closed and bounded
sets are compact.
Properness always implies completeness,
and the following special case of \cite[Proposition~3.4]{BBsemilocal}
shows that the converse holds if $\mu$ is semilocally doubling. 
It  is also shown therein that the constant $\tfrac{2}{3}$ 
is sharp.

\begin{prop} \label{prop-tot-bdd-2/3}
If $\mu$ is doubling within $B(x_0,r_0)$
in the sense of Definition~\ref{def-local-intro}
then $B(x_0,\de r_0)$ 
is totally bounded for every $\de< \tfrac{2}{3}$. 

In particular, if $\mu$ is semilocally doubling 
then $X$ is proper if and only if it is complete.
\end{prop}

Thus, under semilocal doubling, $\Xhat$ is always proper and a local
$(q,p)$-Poincar\'e inequality on $\Xhat$ implies a semilocal one,
whenever $\Xhat$ is connected.

\section{Extensions of Newtonian functions to
\texorpdfstring{$\Xhat$}{X}}

\label{sect-Xhat}

\emph{Recall that from now on it is required that $p>1$.}

\medskip

If a function $u:\Xhat \to \eR$ has a (\p-weak) upper gradient $g$ on $\Xhat$,
then clearly $g|_X$ is a (\p-weak) upper gradient of $u|_X$.
The converse is not true in general, as seen e.g.\ in
$X=\R\setm\Q\subset\R=\Xhat$,
but we will prove the following extension result.

\begin{thm} \label{thm-Xhat-semi}
Assume that the doubling property
and the \p-Poincar\'e inequality hold within the ball $B_0$ in the sense of
Definitions~\ref{def-local-intro}
and~\ref{def-PI}.
Let $\Om\subset B_0$ be open and $u \in \Dp(\Om)$.

Then there is 
$\uhat\in\Dp(\Omhat)$ such that $\uhat=u$ $\CpX$-q.e.\ in 
$\Om$ and the minimal \p-weak upper gradient
$g_{\uhat}$ of $\uhat$ with respect to $\Xhat$ satisfies
\begin{equation}   \label{eq-Xhat-local}
  g_{\uhat} \le A_{0} g_u \quad \text{a.e.\ in }\Om,
\end{equation}
where $A_{0}$ is a constant 
depending only on $p$,
the doubling 
constant and both constants in the \p-Poincar\'e inequality within $B_0$.

If $\mu$ is semilocally doubling and supports a semilocal \p-Poincar\'e 
inequality, 
then the conclusion of the theorem holds for 
all bounded open $\Om\subset X$.
Under global assumptions, the conclusion holds also for unbounded $\Om$
and $A_0$ depends only on 
$p$,
the global doubling 
constant and both constants in the global \p-Poincar\'e inequality.

Moreover, if $\Om$ is \p-path open in $\Xhat$ 
then we can, in the above conclusions, take
$\uhat=u$ in $\Om$ and $g_{\uhat}=g_u$ a.e.\ in $\Om$.
When $\mu$ is semilocally doubling 
and supports a semilocal \p-Poincar\'e inequality,
the extension result holds also for unbounded
open $\Om \subset X$, which are \p-path open in $\Xhat$. 
\end{thm}

A set $\Om\subset \Xhat$ is \p-\emph{path open}\/ in $\Xhat$ if for \p-almost
every curve $\ga:[0,l_\ga]\to\Xhat$, the set
$\ga^{-1}(\Om)$ is relatively open in $[0,l_\ga]$.
By Shanmugalingam~\cite[Remark~3.5]{Sh-harm}, $\Om$
is \p-path open in $\Xhat$ if it is quasiopen in $\Xhat$; 
see also Bj\"orn--Bj\"orn--Mal\'y~\cite{BBMaly} for the
converse implication under certain assumptions.
(The set $\Om$ is \emph{quasiopen} in $\Xhat$ if for every $\eps>0$ there
is an open set $G\subset \Xhat$ such that $\CpXhat(G)<\eps$ and $\Om \cup G$ is open.) 
Note that if $\mu$ is locally doubling,  then 
$X$ (and thus $\Om$) 
is open in $\Xhat$ if and only if it is locally compact. 

For locally compact $X$ with global assumptions, the extension result
$\uhat=u$ in $X$ with $g_\uhat=g_u$ a.e.\ in $X$  was for $u\in\Np(X)$
proved in Lemma~8.2.3 in 
Heinonen--Koskela--Shanmugalingam--Tyson~\cite{HKSTbook}.
A similar result in 
Aikawa--Shan\-mu\-ga\-lin\-gam~\cite[Proposition~7.1]{AikSh05}
relies (via Cheeger~\cite[Theorems~6.1 and~17.1]{Cheeg}) on Cheeger's results,
which assume that $X$ is complete.

\begin{remark}
By Proposition~\ref{prop-Leb-pt} below, $\uhat$ 
in Theorem~\ref{thm-Xhat-semi} may be defined by
\[ 
    \uhat(x)=\limsup_{r \to 0} \vint_{\Bhat(x,r) \cap \Om} u \, d\mu, \quad
x \in \Omhat.
\] 
\end{remark}

The simple example $X=\R \setm \{0\}$ with $u=\chi_{(0,\infty)}$ 
shows that the requirement $\Om\subset B_0$ in 
Theorem~\ref{thm-Xhat-semi} cannot be omitted.
It also demonstrates that in general, under local assumptions, functions
in $\Dp(\Om)$ may fail to have extensions even to $\Dp\loc(\Omhat)$.
A partial remedy for this situation is provided by Lemma~\ref{lem-ext-Nploc}
below for functions from $\Dploc(X)$.

The following example shows that it is essential to require
a Poincar\'e inequality on $X$ in Theorem~\ref{thm-Xhat-semi}.

\begin{example}   \label{ex-R-setm-Q}
Let $X =\R \setm \Q$ equipped with the Lebesgue measure $\mu$,
which is globally doubling on $X$.
As $X$ is totally disconnected, 
 $g_u = 0$ a.e.\ for every $u \in L^p(X)$ and hence $\Np(X)=L^p(X)$.
Thus no Poincar\'e inequality is supported on $X$,
and there is no extension result to $\Xhat$ similar to 
Theorem~\ref{thm-Xhat-semi}.
\end{example}

A natural question is whether the constant $A_0$ in Theorem~\ref{thm-Xhat-semi}
can be chosen equal to one when $\Om$ 
is not \p-path open in $\Xhat$.
Example~\ref{ex-R-setm-Q} shows that it can happen that 
$g_{u,X} =0 < g_{u,\Xhat}$ a.e., but such 
$X$ does not support any Poincar\'e inequality, even though $\mu$ is  
globally doubling on $X$.
On the other hand, the usage of
Proposition~3.5 from Bj\"orn--Bj\"orn~\cite{BBnonopen} 
at the end of the proof of Theorem~\ref{thm-Xhat-semi} shows that
$A_0=1$ also when $\Om$ is only \emph{\p-path almost open}
in $\Xhat$, i.e.\ when
for \p-almost every curve $\ga:[0,l_\ga]\to\subset\Xhat$, the set
$\ga^{-1}(\Om)$ is a union of a relatively open set in $[0,l_\ga]$ and 
a set of 1-dimensional Hausdorff measure zero.

Note, however, that this relaxed assumption is not enough to guarantee
that $\uhat$ can be chosen equal to $u$ everywhere in $\Om$.
This is because in \p-path almost open $\Om$, it can happen that 
there are much fewer zero sets for the capacity $\CpXhat$ 
than for the smaller capacity $\CpOm$.
For example, every $U\subset\R$ with zero 1-dimensional Hausdorff measure 
is \p-path almost open in $\R$ but, as it is totally disconnected, 
we see that $\Cp^U$ is trivial while $\Cp^\R(U)$ can be positive.

\begin{openprob}   \label{open-pr}
Under the assumptions in Theorem~\ref{thm-Xhat-semi}
(and without assuming that $\Om$ is \p-path almost
open in $\Xhat$)
can it happen that it is not true that $g_{\uhat}=g_u$ a.e.? 
\end{openprob}

\begin{proof}[Proof of Theorem~\ref{thm-Xhat-semi}]
In this proof, $C$ will denote various constants which 
only depend on the constants in the local assumptions, and
which may change even within the same line.
Assume to start with that $u$ is bounded.
Let $\eps_k$ be a sequence decreasing to 0 as $k\to\infty$. 
Lemmas~5.1 and~5.2 in Heikkinen--Koskela--Tuominen~\cite{HKT} 
(or a standard Whitney type construction)
provide us, for each $k$, with a cover $\{\Bik\}_{i}$ of $\Om$ by balls 
$\Bik$ of radii
$\rik\le\eps_k$ and a subordinate Lipschitz partition of unity 
$\{\phiik\}_{i}$
so that
\begin{itemize}
\setlength{\itemsep}{0pt}%
\setlength{\parskip}{0pt plus 1pt}%
\item $10\la\Bik\subset \Om$ for all $i$ and $k$;
\item each $10\la\Bik$ meets at most $M$ balls $10\la\Bjk$, and in that case
$\rjk\le2\rik$;
\item each $\phiik$ is $C/\rik$-Lipschitz and vanishes outside $2\Bik$;
\item $\sum_{i}\phiik=1$ in $\Om$.
\end{itemize}
Here $\la$ denotes the dilation constant in the local \p-Poincar\'e inequality 
within $B_0$.
Lemma~5.3 in~\cite{HKT} and its proof (note that $\Bjk\subset 10\Bik$ whenever
$2\Bjk\cap2\Bik\ne\emptyset$) then show that the functions
\[
u_k:= \sum_i u_{\Bik}\phiik
\]
satisfy $u_k\to u$ in $L^p(\Om)$, as $k\to\infty$, and moreover
\begin{alignat}{2}  \label{eq-lip-for-uk}
|u_k(x)-u_k(y)| & \le \frac{Cd(x,y)}{\rik} \vint_{10\Bik} |u-u_{10\Bik}|\,d\mu \\
&\le Cd(x,y) \biggl( \vint_{10\la\Bik} g_u^p\,d\mu \biggr)^{1/p}
&\quad& \text{for all }x,y\in\Bik. \nonumber
\end{alignat}
By passing to a subsequence, we can in addition assume that $u_k\to u$
a.e.\ in $\Om$.

Strictly speaking, $u_k$ are to start with only defined on $\Om$, but the
functions $\phiik$, being Lipschitz, extend uniquely to $\Omhat$
and thus, so do $u_k$.
Call these extensions $\uhat_k$.
Then~\eqref{eq-lip-for-uk} holds for $\uhat_k$ and all $x,y\in\Bhatik$ as well.
Let 
\[ 
 \Lip \uhat_k(x) := \limsup_{r\to0} \sup_{y\in \Bhat(x,r)} \frac{|\uhat_k(y)-\uhat_k(x)|}{r}
\] 
be the \emph{upper pointwise dilation} of $\uhat_k$ (also
called the local upper Lipschitz constant).
It follows from~\eqref{eq-lip-for-uk}
that the minimal \p-weak upper gradient $g_{\uhat_k}$ 
(with respect to $\Xhat$) 
satisfies
\begin{equation}    \label{eq-est-g-uhat-k}
g_{\uhat_k}(x) \le \Lip\uhat_k(x)
\le C \biggl( \vint_{10\la\Bik} g_u^p\,d\mu \biggr)^{1/p}
\quad \text{for a.e.\ $x\in\Bhatik$}, 
\end{equation}
see Proposition~1.14 in~\cite{BBbook}. 
Since $\mu(\Xhat\setm X)=0$ and the balls $10\la\Bik$ have bounded overlap, 
this implies that
\[ 
\int_{\Omhat} g_{\uhat_k}^p \,d\mu 
\le C \sum_i  
   \int_{\Bik} \biggl( \vint_{10\la\Bik} g_u^p\,d\mu \biggr) \,d\mu 
\le C \int_{\Om} g_u^p\,d\mu.
\]
We can therefore conclude from Lemma~6.2 in~\cite{BBbook}
that there is a subsequence of $\uhat_k$
(also denoted $\uhat_k$) converging weakly in 
$L^p(\Omhat)$ to some 
$\uhat\in L^p(\Omhat)$ and such that $g_{\uhat_k}\to g$ 
weakly in $L^p(\Omhat)$,
where $g\in L^p(\Omhat)$ is a \p-weak upper gradient 
(with respect to $\Xhat$) of $\uhat$.
Moreover, $\uhat\in\Np(\Omhat)$  and
$\uhat=u$ a.e.\ in $\Om$.
Hence also $\uhat=u$ $\CpX$-q.e.\ in $\Om$,
since $u,\uhat\in\Np(\Om)$.

The Lebesgue differentiation theorem holds in $B_0$, 
cf.\ \cite[Theorem~3.9]{BBsemilocal}.
Let $x\in \Om$ be a Lebesgue point of both $g$ and $g_u^p$.
Then for each $\eps>0$ there exists $\rho_0>0$ such that for every
$B=B(x,\rho)$ with $0<\rho<\rho_0$,
\[
|g(x)- g_{B}|<\eps \quad \text{and} \quad |g_u^p(x)- (g_u^p)_{B}|<\eps.
\]
We then have by~\eqref{eq-est-g-uhat-k} and the weak convergence of $g_{\uhat_k}$
that
\begin{align*}
g(x)-\eps &< \vint_{B} g\,d\mu \\
  & = \lim_{k\to\infty}  \vint_{B} g_{\uhat_k}\,d\mu \\
   &\le \liminf_{k\to\infty}  \frac{C}{\mu(B)} \sum_{\Bik\cap B\ne\emptyset}
 \mu(\Bik) \biggl( \vint_{10\la\Bik} g_u^p\,d\mu \biggr)^{1/p} \\
   &
  \le \liminf_{k\to\infty}  \frac{C}{\mu(B)} \sum_{\Bik\cap B\ne\emptyset}
 \mu(\Bik)^{1-1/p} \biggl( \int_{10\la\Bik} g_u^p\,d\mu \biggr)^{1/p}.
\end{align*}
The last expression 
can be estimated using the H\"older inequality and the bounded overlap
of the balls $10\la\Bik$.
We therefore conclude that
\begin{align*}
g(x)-\eps &< \liminf_{k\to\infty} \frac{C}{\mu(B)} 
\biggl( \sum_{\Bik\cap B\ne\emptyset} \int_{10\la\Bik} g_u^p\,d\mu \biggr)^{1/p}
\biggl( \sum_{\Bik\cap B\ne\emptyset} \mu(\Bik) \biggr)^{1-1/p} \\
&\le C \liminf_{k\to\infty} \biggl( \vint_{B(x,\rho+20\la\eps_k)} g_u^p
       \,d\mu \biggr)^{1/p}
\le C(g_u(x)+\eps).
\end{align*}
Letting $\eps\to0$ proves the first part of the theorem for bounded $u$
and $\Om$.
For unbounded $u$, use the truncations $\min\{k,\max\{u,-k\}\}$ of $u$ 
at $\pm k$.

If $\Om$ is unbounded and $\mu$ is globally doubling and supporting 
a global \p-Poincar\'e inequality,
then we apply the above arguments 
to the  
sets $\Om_k=\Om\cap B(x_0,k)$.
More precisely, by the above we can find
$\uhat_1 \in \Dp(\Omhat_1)$ such that $\uhat_1=u$  
$\CpX$-q.e.\ on $\Om_1$.
We can also find $\uhat_2 \in \Dp(\Omhat_2)$ such that $\uhat_2=u$ 
$\CpX$-q.e.\ in $\Om_2$.
As the set $\{y \in \Omhat_1 : \uhat_1(y) \ne \uhat_2(y)\}$
has measure zero, it must be of zero $\CpXhat$-capacity, and 
we are thus free to choose $\uhat_2=\uhat_1$ in $\Omhat_1$.
Proceeding in this way, we can construct $\uhat \in \Dp(\Omhat)$ 
so that $\uhat=u$ $\CpX$-q.e.\ on $X$.
Moreover, 
$g_{\uhat} \le A g_u$ a.e.\ in $\Om$,
where $A$ only depends on $p$, the global
doubling constant and the constants in the global \p-Poincar\'e inequality.

If $\Om$ is \p-path open in $\Xhat$ then the capacities 
$\CpOm$ and $\CpXhat$ have the same zero sets in $\Om$,
by Proposition~4.2 in Bj\"orn--Bj\"orn--Mal\'y~\cite{BBMaly}.
By Lemma~2.24 in \cite{BBbook} the zero sets are also
the same for $\CpOm$ and $\CpX$ for sets in $\Om$.
This shows that we may choose $\uhat=u$ in $\Om$.
That the  
minimal \p-weak upper gradients with respect to $X$ 
and $\Xhat$ are equal follows from
Proposition~3.5 in Bj\"orn--Bj\"orn~\cite{BBnonopen}.
In this case, the argument above for unbounded $\Om$ also  holds
under semilocal assumptions, since 
$A\equiv1$.
\end{proof}

The extension Theorem~\ref{thm-Xhat-semi} makes it possible to
obtain several qualitative results about Newtonian functions on
noncomplete spaces under local assumptions.
For this, it is even enough that the functions belong to the local spaces.
The following two lemmas will therefore be useful.

\begin{lem} \label{lem-Nploc=Dploc}
If $X$ supports a local $(p,p)$-Poincar\'e inequality
{\rm(}or if $\mu$ is locally doubling and supports a local \p-Poincar\'e 
inequality{\rm)} then $\Nploc(X)=\Dploc(X)$.
\end{lem}

\begin{proof}
By Theorem~5.1 
in \cite{BBsemilocal} we can assume that
$X$ supports a local $(p,p)$-Poincar\'e inequality, from which the
result now follows as in the proof of \cite[Proposition~4.14]{BBbook}.
\end{proof}

\begin{lem} \label{lem-ext-Nploc}
Assume that 
$\mu$ is locally doubling and supports a local \p-Poincar\'e 
inequality.
Then for every $u \in \Nploc(X)$ there is an open set
$\Ghat \supset X$ in $\Xhat$ 
and a function $\uhat \in \Nploc(\Ghat)$ such that $u=\uhat$ 
$\CpX$-q.e.\ on $X$.
Moreover, $\Ghat$ is locally compact and $\mu|_{\Ghat}$ is locally doubling
and supports a local \p-Poincar\'e inequality. 

If $X$ is \p-path open in $\Xhat$, then one can choose $\uhat=u$
and $g_\uhat=g_u$
everywhere in $X$. 
\end{lem}

\begin{proof}
For each $x \in X$ we can find a ball $B(x,r_x)$ such that
the \p-Poincar\'e inequality and the 
doubling property for $\mu$ hold within $B(x,r_x)$, 
and such that
$u \in \Np(B(x,r_x))$.
As $X$ is Lindel\"of, we can find a countable cover $\{B_j\}_{j=1}^\infty$
of $X$, where $B_j=B(x_j,r_{x_j})$.  
Let $\Bhat_j=\Bhat(x_j,r_{x_j})$ 
and $\Ghat=\bigcup_{j=1}^\infty \Bhat_j$.

By Theorem~\ref{thm-Xhat-semi}, we can find $\uhat_1 \in \Np(\Bhat_1)$ 
such that $\uhat_1=u$  
$\CpX$-q.e.\ on $B_1$.
We can also find $\uhat_2 \in \Np(\Bhat_1\cup \Bhat_2)$ 
such that $\uhat_2=u$ 
$\CpX$-q.e.\ on $B_1 \cup B_2$.
As the set $\{y \in \Bhat_1 : \uhat_1(y) \ne \uhat_2(y)\}$
has measure zero, it must be of zero $\CpXhat$-capacity, and thus we 
are free to choose $\uhat_2=\uhat_1$
on $\Bhat_1$.
Proceeding in this way, we can construct $\uhat \in \Nploc(\Ghat)$ 
so that $\uhat=u$ $\CpX$-q.e.\ on $X$.
Note that, by construction,
\begin{equation}  \label{eq-int-Bhat-A0}
\int_{\Bhat_j} g_\uhat^p\,d\mu \le A_j\int_{B_j} g_u^p\,d\mu,
\end{equation}
where $A_j$ is the constant provided by Theorem~\ref{thm-Xhat-semi} on $B_j$.
If $X$ is \p-path open in $X$, then it follows from the last part of 
Theorem~\ref{thm-Xhat-semi} that we can choose 
$\uhat=u$ everywhere in $X$
and $A_j\equiv 1$, i.e.\ $g_\uhat=g_u$.

The local doubling property and the local \p-Poincar\'e inequality
for $\mu|_{\Ghat}$ follow from
Propositions~\ref{prop-X-Xhat-doubl} and~\ref{prop-X-Xhat-PI}.
Consequently, each $\Bhat_j$ (and thus also $\Ghat$)
is locally compact, by Proposition~\ref{prop-tot-bdd-2/3}.
\end{proof}

As the local assumptions are inherited by open subsets of $X$, 
Lemma~\ref{lem-ext-Nploc} can be directly applied to 
them as well.
Note that the set $\Ghat$ depends on $u$.
The following example shows that this drawback cannot be avoided.

\begin{example}   \label{ex-G-dep-u}
Let $B$ be a ball in $\R^n$ and $Z=\{z_j\}_{j=1}^\infty$ be a dense subset 
of $B$. 
Set $X=B\setm Z$, equipped with the Lebesgue measure. 
Note that $\Xhat=\itoverline{B}$.

If $p<n$ then $\Cp(Z)=0$ and hence there are \p-almost no curves in
$\R^n$ passing through $Z$.
It follows that \p-weak upper gradients with respect to $X$ and $\R^n$
are the same for every measurable function $u:X\to\eR$ (extended
arbitrarily on $Z$).
Thus, $X$ supports a global \p-Poincar\'e inequality (and is, of course,
globally doubling).

Now, for each $j=1,2,\ldots$, the function $u_j(x)=|x-z_j|^\al$, $\al\in\R$,
belongs to $\Np\loc(X)$.
However, for $\al\le1-n/p$ it can only extend to a function in
$\Np\loc(\Xhat\setm\{z_j\})$, not in $\Np\loc(\Xhat)$.
This shows that the set $\Ghat$ in Lemma~\ref{lem-ext-Nploc} indeed
must depend on~$u$.
\end{example}

The following two results are now relatively easy consequences
of the above extensions to $\Xhat$ and 
the corresponding results in complete spaces
from \cite{BBsemilocal}.

\begin{prop} \label{prop-Leb-pt}
Assume that  
$\mu$ is locally doubling and supports a local \p-Poincar\'e 
inequality.
Then every $u \in \Nploc(X)$ has $L^p$-Lebesgue points
$\CpX$-q.e.,
and moreover the extension $\uhat$ in Lemma~\ref{lem-ext-Nploc}
can be given by
\begin{equation} \label{eq-Leb-pt}
    \uhat(x)=\limsup_{r \to 0} \vint_{\Bhat(x,r) \cap X} u \, d\mu,
   \quad x \in \Ghat.
\end{equation}
\end{prop}

See Remark~7.2 in \cite{BBsemilocal} for  
further discussion on
$L^q$-Lebesgue points for $q>p$.
Note that the proof below shows that the limit 
\[
\lim_{r \to 0} \vint_{B(x,r)} u \, d\mu
\]
actually exists for $\CpXhat$-q.e.\ $x\in X$, even though it only equals $u(x)$
for $\CpX$-q.e.\ $x$.

\begin{proof}
Find $\Ghat$ and $\uhat$ as in Lemma~\ref{lem-ext-Nploc}.
It then follows from \cite[Theorem~7.1]{BBsemilocal} 
that $\uhat$ has $L^p$-Lebesgue points $\CpXhat$-q.e.\ in $\Ghat$.
As $u=\uhat$ $\CpX$-q.e.\ in $X$, we conclude that $u$ has $L^p$-Lebesgue 
points $\CpX$-q.e.\ in $X$.

Finally, if $\ut$ is given by \eqref{eq-Leb-pt}, then $\uhat=\ut$ 
at all $L^1$-Lebesgue points of $\uhat$, i.e.\ $\CpXhat$-q.e.\ in $\Ghat$. 
Hence,
$\uhat$ may also be chosen so that it satisfies \eqref{eq-Leb-pt}.
\end{proof}

\begin{proof}[Proof of Theorem~\ref{thm-Xhat-intro}]
This result follows directly from
Theorem~\ref{thm-Xhat-semi} and Proposition~\ref{prop-Leb-pt}.
\end{proof}

\begin{prop}  \label{prop-qcont}
Assume that  
$\mu$ is locally doubling and supports a local \p-Poincar\'e 
inequality, and that $X$ is \p-path open in $\Xhat$.
Then every $u \in \Nploc(X)$ is quasicontinuous.
\end{prop}

A function $u$ is \emph{quasicontinuous} on $X$
if for every $\eps>0$ there is an open $G \subset X$ such that
$\CpX(G)<\eps$ and $u|_{X \setm G}$ is real-valued and continuous.

Quasicontinuity has earlier been established for Newtonian functions
under various assumptions in 
Bj\"orn--Bj\"orn--Shanmugalingam~\cite{BBS5}, 
Heinonen--Koskela--Shan\-mu\-ga\-lin\-gam--Tyson~\cite{HKSTbook},
Bj\"orn--Bj\"orn--Lehrb\"ack~\cite{BBLeh1}
and in \cite{BBsemilocal} for open sets in locally compact spaces.
Existence of quasicontinuous representatives under global
assumptions was obtained already in Shanmugalingam~\cite{Sh-rev}.
Assuming completeness and global assumptions, quasicontinuity 
can be proved also on quasiopen  sets, see
Bj\"orn--Bj\"orn--Latvala~\cite{BBLat3} and
Bj\"orn--Bj\"orn--Mal\'y~\cite{BBMaly}.

\begin{proof}
Find $\Ghat$ and $\uhat$ as in Lemma~\ref{lem-ext-Nploc},
with $\uhat=u$ in $X$.
It then follows from \cite[Theorem~9.1]{BBsemilocal}
that $\uhat$ is quasicontinuous on $\Ghat$, which immediately yields
that $u$ is quasicontinuous on $X$, since $\CpX$ is dominated by
$\CpXhat$.
\end{proof}

As a direct consequence of Proposition~\ref{prop-qcont} 
we can also conclude from \cite[Theorem~5.31]{BBbook} 
that $\CpX$ is an outer (and Choquet) capacity on $X$.
Moreover, by \cite[Theorem~8.4 and Proposition~9.3]{BBsemilocal},
if $K \subset X$ is compact, then
\[
    \CpX(K) = \inf
\|u\|_{\Np(X)}^p,
\]
where the infimum is taken over all 
locally  Lipschitz
$u$ such that $u\ge1$ on $K$.

\begin{remark}
Even for $u \in \Np(X)$, 
Lemma~\ref{lem-ext-Nploc} only guarantees an extension in the local 
Newtonian space $\Nploc(\Ghat)$.
The set $\Ghat$ can, however, be chosen independently of $u$,
since the covering balls can be chosen so, when $u\in\Np(X)$.
In general, we do not know if it is possible to find an extension
in $\Np(\Ghat)$, 
since we lack a uniform control of the constant $A_0$ in 
Theorem~\ref{thm-Xhat-semi}, and thus in~\eqref{eq-int-Bhat-A0}.
However, this can be achieved in the following situations 
(which can also be combined on different parts of $X$):
\begin{enumerate}
\item
One can find a \emph{finite} cover by balls $B_j$ as in the proof
of Lemma~\ref{lem-ext-Nploc}.
\item
Each ball $B_j$ is \p-path almost open in $\Xhat$, 
which guarantees that $A_0\equiv1$.
\item
$\mu$ is both locally doubling and supports a 
local \p-Poincar\'e inequality with uniform  
constants independent of $x_0$ and $r_0$, 
which guarantees that $A_0$ is uniformly bounded.
\end{enumerate}

As a matter of fact,
as discussed 
just before Open problem~\ref{open-pr},
it is not known if the constant $A_0$ 
in Theorem~\ref{thm-Xhat-semi}
ever needs to be larger than $1$.
\end{remark}

\section{Self-improvement of Poincar\'e inequalities}
\label{sect-self-imp-PI}

A deep result due to Keith--Zhong~\cite[Theorem~1.0.1]{KZ}
shows that the  Poincar\'e inequality is an open-ended property.
See also 
Heinonen--Koskela--Shanmugalingam--Tyson~\cite[Theorem~12.3.9]{HKSTbook} and 
Eriksson--Bique~\cite{Eriksson-Bique}.
By localizing the arguments in~\cite{HKSTbook},
the following local version of the self-improvement result
was obtained in \cite[Theorem~5.3]{BBsemilocal}.

\begin{thm} \label{thm-KZ-proper}
Let $B_0=B(x_0,r_0)$ be a ball such that 
$\itoverline{B}_0$ is compact and
the \p-Poincar\'e inequality and the doubling property for $\mu$ hold
within $B_0$ in the sense of 
Definitions~\ref{def-local-intro}
and~\ref{def-PI}.

Then there exist constants $C$, $\la$ and $q<p$, depending only on $p$,
the doubling 
constant and both constants in the \p-Poincar\'e inequality within $B_0$,
such that for all balls $B$ with $3\la B\subset B_0$,
all integrable functions $u$ on $\la B$, and all 
$q$-weak upper gradients $g$ of $u$, 
\begin{equation}    \label{eq-KZ-proper}
        \vint_{B} |u-u_B| \,\dmu
        \le C r_B \biggl( \vint_{\la B} g^{q} \,\dmu \biggr)^{1/q}.
\end{equation}
\end{thm}

Theorem~\ref{thm-KZ-proper} relatively easily leads to a local 
self-improvement under local assumptions. 
With a little more work it also yields a semilocal conclusion,
cf.\ \cite[Theorem~5.4]{BBsemilocal}.

In Heinonen--Koskela--Shanmugalingam--Tyson~\cite[Theorem~12.3.10]{HKSTbook}
it is explained how (under global assumptions) the properness of $X$
in Keith--Zhong~\cite{KZ} 
can be relaxed to local compactness, with
somewhat weaker global conclusions (namely that the weak
upper gradients considered therein are required to be $L^p$-integrable).
In fact, using extension Theorem~\ref{thm-Xhat-semi}, 
even local compactness can be 
disposed of, as we shall now see.

\begin{thm} \label{thm-KZ-local}
Let $B_0=B(x_0,r_0)$ be a ball such that 
the \p-Poincar\'e inequality  holds
within $B_0$, while the doubling property for $\mu$ holds within
$\tau B_0$ for some $\tau > \tfrac{3}{2}$, 
in the sense of Definitions~\ref{def-local-intro}
and~\ref{def-PI}.

Then there exist constants $C$, $\la$ and $q<p$, 
depending only on $p$, the doubling 
constant and both constants in the \p-Poincar\'e inequality within $B_0$,
such that for all balls $B$ with $3\la B\subset B_0$,
all integrable functions $u\in \Dp(\la B)$ and all 
\p-weak upper gradients $g\in L^p(\la B)$ of $u$, 
\begin{equation}    \label{eq-KZ}
        \vint_{B} |u-u_B| \,\dmu
        \le C r_B \biggl( \vint_{\la B} g^{q} \,\dmu \biggr)^{1/q}.
\end{equation}
If $B_0$ is, in addition, \p-path almost open in $\Xhat$,
which in particular holds if $X$ is locally compact,
then \eqref{eq-KZ} holds for 
all $q$-weak upper gradients $g$ of $u\in \Dp(\la B)$ in $\la B$.
\end{thm}

Note that, since $u$ is assumed to have an $L^p$-integrable upper
gradient, the latter part of this result \emph{does not} show that 
$X$ supports a (semi)local $q$-Poincar\'e inequality.
Neither does \cite[Proposition~12.3.10]{HKSTbook} imply 
that $X$ supports a global $q$-Poincar\'e inequality.
Koskela~\cite{Koskela} has given counterexamples showing
that this cannot be concluded without completeness.

The proof shows that if it is known that $B_0$ is totally bounded,
then it is enough to require doubling only within $B_0$.

\begin{proof}
Proposition~\ref{prop-tot-bdd-2/3}
implies that the ball $B_0$ is totally bounded
and hence the $\Xhat$-closure 
of $\Bhat_0:=\Bhat(x_0,r_0)$ is compact. 
Propositions~\ref{prop-X-Xhat-doubl} and~\ref{prop-X-Xhat-PI} imply that
the doubling property and the \p-Poincar\'e inequality hold
within $\Bhat_0$, with the same constants.

It then follows from Theorem~\ref{thm-KZ-proper} 
that there exist constants $C$, $\la$ and $q<p$ such that 
the following variant of~\eqref{eq-KZ-proper} 
holds for all balls $\Bhat$ with
$3\la\Bhat \subset \Bhat_0$, 
all integrable functions $\uhat$ on $\la\Bhat$ and all 
$q$-weak upper gradients $\ghat$ of $\uhat$
with respect to~$\la\Bhat$,
\begin{equation}    \label{eq-KZ-local}
        \vint_{\Bhat} |\uhat-\uhat_{\Bhat}| \,\dmu
        \le C r_\Bhat \biggl( \vint_{\la\Bhat} \ghat^{q} \,\dmu \biggr)^{1/q}.
\end{equation}
Now, let $B=B(x,r)$ with $3\la B \subset B_0$ be arbitrary 
and set $\Bhat=\Bhat(x,r)$.
Using Theorem~\ref{thm-Xhat-semi}, we can
for every $u\in\Dp(\la B)$ find $\uhat\in \Dp(\la \Bhat)$,
which is an extension of a representative of $u$ and such that the
minimal \p-weak upper gradients $g_{\uhat}$ and $g_u$ of $\uhat$ and $u$
(with respect to $\la \Bhat$ and $\la B$, respectively) 
satisfy $g_{\uhat} \le A_{0} g_u$ a.e.\ in $\la B$,
where the constant $A_{0}$ depends only 
on $p$ and the doubling and Poincar\'e
constants within $B_0$.
Since $g_u$ and $g_{\uhat}$ are also $q$-weak upper gradients
(by \cite[Proposition~2.45]{BBbook}),
we conclude from~\eqref{eq-KZ-local} that
\begin{align}    \label{eq-from-Xhat-to-X}
        \vint_{B} |u-u_B| \,\dmu
        &= \vint_{\Bhat} |\uhat-\uhat_{\Bhat}| \,\dmu 
        \le C r \biggl( \vint_{\lambda \Bhat} g_{\uhat}^{q} \,\dmu 
                   \biggr)^{1/q} \\
    &     \le CA_{0} r \biggl( \vint_{\lambda B} g_{u}^{q} \,\dmu 
           \biggr)^{1/q}
         \le CA_{0} r \biggl( \vint_{\lambda B} g^{q} \,\dmu 
           \biggr)^{1/q}, \nonumber
\end{align}
whenever $g\in L^p(\la B)$ is a \p-weak upper gradient of $u$ 
(although not necessarily for upper gradients $g \in L^q(\la B)$, 
since they need not extend to $\la\Bhat$). 
This proves~\eqref{eq-KZ}.

For the last part, assume that $B_0$ is, in addition, 
\p-path almost open in $\Xhat$, 
and that $g$ is a $q$-weak upper gradient
of $u$ in $\la B$ such that the right-hand side in \eqref{eq-KZ} is finite.
Then 
$g \ge g_{u,\la B,q}$ a.e., where $g_{u,\la B,q}$ is the minimal
$q$-weak upper gradient of $u$ in $\la B$.
Since $B_0$ is \p-path almost open in $\Xhat$, 
it is easily verified that $\la B$ is 
\p-path almost open in $\Xhat$, 
and hence
also $q$-path almost  open, by \cite[Proposition~2.45]{BBbook}.
Proposition~3.5 in Bj\"orn--Bj\"orn~\cite{BBnonopen} then shows that
\[
    g_{u,\la B,q} 
=  g_{\uhat,\la\Bhat,q} 
    \quad \text{a.e. in } \la B.
\]
Thus, by~\eqref{eq-KZ-local} 
again,
\begin{align*}
        \vint_{B} |u-u_B| \,\dmu
        &= \vint_{\Bhat} |\uhat-\uhat_{\Bhat}| \,\dmu \\
        &  \le C r \biggl( \vint_{\lambda \Bhat} g_{\uhat,\la\Bhat,q}^q \,\dmu \biggr)^{1/q}
        \le C r \biggl( \vint_{\lambda B} g^{q} \,\dmu \biggr)^{1/q}.\qedhere
\end{align*}

\end{proof}

\begin{cor}
If $\mu$ is locally doubling and supports a 
local \p-Poincar\'e inequality, 
then for every $x_0\in X$ there is a ball $B_0' \ni x_0$,
together with constants $C$, $\la$ and $q<p$, such that
\eqref{eq-KZ} holds for all balls $B\subset B_0'$
\textup{(}not just for $3\la B \subset B_0'$\textup{)},
all integrable functions $u\in \Dp(\la B)$ and all 
\p-weak upper gradients $g\in L^p(\la B)$ of $u$.

If the assumptions about doubling and \p-Poincar\'e inequality
are semilocal, then the conclusion of the theorem is also 
semilocal, i.e.\ it holds for all balls $B_0'\subset X$. 
Under global assumptions, the constants $C$, $\la$ and $q$
are independent of $B_0$, i.e.\ the conclusion is global.

If $X$ is, in addition, \p-path almost open in $\Xhat$, which 
in particular holds if $X$ is locally compact,
then \eqref{eq-KZ} holds for 
all $q$-weak upper gradients $g$ of $u\in \Dp(\la B)$ in $\la B$.
\end{cor}

\begin{proof}
Let $x_0\in X$ be arbitrary and find $r_0>0$ so that the assumptions
of Theorem~\ref{thm-KZ-proper} hold for $B_0=B(x_0,r_0)$.
Then choose a radius $0<r_0'\le (7\la)^{-1}r_0$ so that 
$B_0':=B(x_0,r_0') \ne X$ and $\dist(x_0,X \setm B_0')=r_0'$.
For $B\subset B_0'$ it then follows that $r_B\le2r_0'$ and hence 
$3\la B \subset B(x_0,r_0)$.
The first statement then follows from Theorem~\ref{thm-KZ-local}.

Since $\la$ in Theorem~\ref{thm-KZ-local} depends on $B_0$, we cannot
directly obtain a semilocal conclusion 
(under semilocal assumptions) from it.
However, under global assumptions, the constants $C$, $\la$, $q<p$ 
and $A_0$ will be independent of $B_0$, which yields the global result.

To reach a semilocal conclusion under semilocal assumptions, 
we instead note that
$\Xhat$ is proper and connected, by
Proposition~\ref{prop-tot-bdd-2/3} and 
the proof of \cite[Proposition~4.2]{BBbook}.
Theorem~5.4 in \cite{BBsemilocal} then implies that for every ball 
$B_0'=B(x_0,r_0')\subset X$
there exist constants $C$, $\la$ and $q<p$, such that
\eqref{eq-KZ-local} holds for all balls $\Bhat\subset \Bhat(x_0,r_0')$,
all integrable functions $\uhat$ on $\la\Bhat$ and all 
$q$-weak upper gradients $\ghat$ of $\uhat$
with respect to~$\la\Bhat$.
By enlarging $r_0'$ if necessary, we may 
assume that $\dist(x_0,X \setm B_0')=r_0'$. 
(If $B_0'=X$, we instead note that $X$ is bounded and thus 
semilocal assumptions are the same as global assumptions,
which were handled above.)
If now $B\subset B_0'$ is a ball then $r_B\le2r_0'$ and hence 
$\la B \subset (1+2\la)B_0'=:B_0$.
Theorem~\ref{thm-Xhat-semi}, applied to $B_0$ and followed by 
\eqref{eq-from-Xhat-to-X}, then yields~\eqref{eq-KZ}.

The last part about $q$-weak upper gradients
in the \p-path almost open case
follows as in the (last part of the)
proof of Theorem~\ref{thm-KZ-local}.
\end{proof}

\section{\texorpdfstring{\p}{p}-harmonic functions in noncomplete spaces}

In this section we conclude the paper with a discussion on possible 
directions for developing the theory of \p-harmonic functions and
quasiminimizers on noncomplete spaces.

\emph{Let $\Om\subset X$ be open throughout this section.}
Traditionally, e.g.\ in $\R^n$ and other complete spaces, \p-harmonic
functions on $\Om$ are required to belong to the local space $\Np\loc(\Om)$
and their \p-harmonicity is tested by sufficiently smooth (e.g.\ 
Lipschitz or Sobolev) functions $\phi$ with compact support in $\Om$
(or with zero boundary values) as follows:
\begin{equation} \label{eq-deff-p-harm}
      \int_{\phi\ne0} g^p_u \, d\mu 
           \le \int_{\phi\ne0} g_{u+\phi}^p \, d\mu.
\end{equation}
For practical applications it can then often be shown that the \p-harmonicity
can equivalently be tested by other classes of test functions as well.
Let us have a closer look at these spaces.
In Section~\ref{sect-prelim}, we defined $\Np\loc(\Om)$ as the space 
of all functions $u$ such that 
\[
\text{for every $x\in\Om$ there exists a ball $B_x\ni x$ such that
$u\in\Np(B_x)$.}
\]
It is an easy exercise to see that if $\Om$ is locally compact then this 
definition is equivalent to the requirement that 
\[
\text{$u\in\Np(G)$ for all
open sets $G$ such that $\clG$ is a compact 
subset of $\Om$.}
\]
Note that $\Om$, being an open subset of $X$, is always locally compact if
$X$ is proper.
Also recall that, by Proposition~\ref{prop-tot-bdd-2/3}, if $\mu$ is
semilocally doubling then $X$ is proper if and only if it is complete.

In noncomplete spaces, defining $\Np\loc(\Om)$ through compact subsets of
$\Om$ might not be so useful, since there may be 
no (or very few) nonempty open sets with compact closures.
The same applies to the definitions of the space of test functions
in~\eqref{eq-deff-p-harm}.
We therefore consider the following families of bounded open subsets 
of $\Om$:
\begin{align*}
\Gcpt &= 
\{\text{bounded open }G\subset\Om: \clG \subset\Om \text{ is compact} \} \\
\Gdist &= 
\{\text{bounded open }G\subset\Om: \dist(G,X\setm\Om)>0 \} \\
\Gbdy &= 
\{\text{bounded open }G\subset\Om: \dist(G,\bdry\Om)>0 \} \\
\Gclos &= 
\{\text{bounded open }G\subset\Om: \clG \subset\Om \}. 
\end{align*}
(Here we consider $\dist(G,\emptyset)>0$.)

It is easily verified that
\[
\Gcpt \subset \Gdist \subset \Gbdy \subset \Gclos.
\]
Hence, if the local Newtonian space $\Nxloc(\Om)$, with 
${\bf x} \in \{\text{cpt}, \text{dist}, \text{bdy},\text{clos}\}$,
is defined by
\[
\Nxloc(\Om) = \{u:\Om\to\eR: u\in \Np(G) \text{ for all } G\in \Gx \},
\]
then
\[
\Nclosloc(\Om) \subset \Nbdyloc(\Om) \subset \Ndistloc(\Om) \subset 
\Np\loc(\Om) \subset \Ncptloc(\Om),
\]
where the last two inclusions follow from the fact that every
ball $B(x,r_x)$ with $x\in\Om$ and $r_x<\dist(x,X\setm\Om)$ belongs
to $\Gdist$ and that every compact set can be covered by finitely many
such balls.

If $X$ is proper then clearly $\Gcpt=\Gclos$ and all the above five
local Newtonian spaces coincide,
while the last two spaces always coincide if $X$ is locally compact.
Depending on $X$ and $\Om$, some partial equalities are possible also
in noncomplete spaces, see Example~\ref{ex-slit-plane}.

Now we turn our attention to the spaces of test functions 
in~\eqref{eq-deff-p-harm} and define:
\begin{alignat*}{2}
\Np_0(E) &= \{\phi|_E: \phi\in\Np(X) \text{ and } 
\phi=0 \text{ in }X\setm E\}
&\quad& \text{for } E\subset X,\\
\Nxo(\Om) &= \overline{ \{\phi:\Om\to\eR: \phi\in \Np_0(G) 
\text{ for some } G\in \Gx \}},
\end{alignat*}
where the closure is taken in $\Np(X)$ and functions in $\Np_0(G)$
are regarded as extended by zero outside of $G$. 
Alternatively, only the noncomplete spaces 
\[
\{\phi:\Om\to\eR: \phi\in \Np_0(G) \text{ for some } G\in \Gx \}
\]
could be considered.
Since $\Np_0(\Om)$ is closed in $\Np(X)$ (by \cite[Theorem~2.36]{BBbook}), 
we immediately see that
\begin{equation}   \label{eq-incl-Np0}
\Ncpto(\Om) \subset \Ndisto(\Om) \subset \Nbdyo(\Om) 
\subset \Ncloso(\Om) \subset \Np_0(\Om).
\end{equation}
As before, if $X$ is proper then the first four spaces of test functions
coincide.
If, in addition, all functions in $\Np(X)$ are quasicontinuous 
then \cite[Lemma~5.43]{BBbook}
implies that $\Np_0(\Om)= \Ncloso(\Om)$, 
i.e.\ all the above spaces of test functions coincide.
This in particular holds if $X$ is locally compact and $\mu$ 
is locally doubling and supporting a local \p-Poincar\'e inequality,
by \cite[Theorem~9.1]{BBsemilocal}.

There are also other classes of test functions that one can consider, 
e.g.\ 
\[
    \Npc(\Om) =
   \overline{\{\phi \in \Np(X) : \supp \phi \text{ is a compact subset of } \Om\}},
\]
i.e.\ $\Npc(\Om)$ is defined as $\Ncpto(\Om)$ 
omitting the word ``open'' in $\Gcpt$.
If $X$ (or $\Om$) is locally compact, then $\Npc(\Om)=\Ncpto(\Om)$,
but this is not true in general.
On the other hand, 
since $X$ is a normal topological space, we have
\[
    \Nxo(\Om) = \overline{ \{\phi:\Om\to\eR: \phi\in \Np_0(E) 
\text{ for some } E\in \Ex \}}
 \quad \text{if }{\bf x}  \in \{\text{dist}, \text{bdy},\text{clos}\},
\]
where $\Ex$ is defined as $\Gx$ but omitting the word ``open''.
Test function 
classes based on Lipschitz functions similarly to any of the above classes
are also possible, see Bj\"orn--Marola~\cite[Section~4]{BMarola}.
We will not discuss these classes of test functions further.

Each of the local Newtonian spaces, defined above, 
can appear in the definition of \p-harmonic functions, 
together with one of the above spaces of test functions.

\begin{deff} \label{def-quasimin}
A function $u \in \Nxloc(\Om)$ is a 
\emph{$Q$-quasi\/\textup{(}super\/\textup{)}minimizer} in $\Om$ if 
\begin{equation} \label{eq-deff-qmin}
      \int_{\phi\ne0} g^p_u \, d\mu 
           \le Q \int_{\phi\ne0} g_{u+\phi}^p \, d\mu
\end{equation}
for all (nonnegative/nonpositive) $\phi\in\Nyo(\Om)$.
If $Q=1$ in \eqref{eq-deff-qmin} then $u$ is 
a \emph{\textup{(}super\/\textup{)}minimizer}.
A \emph{\p-harmonic function} is a continuous minimizer.
\end{deff}

Here each of {\bf x} and {\bf y} stands for one of the above defined subscripts,
or the absence of such a subscript in the case of $\Nploc$ and $\Np_0$.
Naturally, different choices of {\bf x} and {\bf y} in 
Definition~\ref{def-quasimin} lead to different classes of quasiminimizers
and \p-harmonic functions, which may have advantages and disadvantages,
depending on the situation and the intended applications.
In Example~\ref{ex-slit-plane} below we demonstrate some of these differences,
but first we show that interior regularity can be obtained for most of these 
definitions.

The largest class of (quasi)minimizers is obtained when allowing for a 
large local Newtonian space and by testing 
with as few test functions as possible. 
This is reflected in the choice of function spaces in
the following regularity result.
To cover also non-locally compact spaces,
we exclude the definitions involving $\Gcpt(\Om)$,
as well as $\Npc(\Om)$.
We say that a function $u:\Om\to\eR$ is \emph{lsc-regularized} if
\[
u(x)=\essliminf_{y \to x} u(y) \quad \text{for all } x \in \Om.
\]

\begin{thm}   \label{thm-semiloc-int-reg}
Assume that $\mu$ is locally doubling and supports a local 
\p-Poincar\'e inequality in $\Om$.
Let $u\in\Np\loc(\Om)$ be a $Q$-quasi\/\textup{(}super\/\textup{)}minimizer 
in $\Om$, tested by $\phi \in \Ndisto(\Om)$.
Then $u$ has a representative $\ut$ which is continuous 
\textup{(}resp.\ lsc-regularized\/\textup{)}.

Moreover, the strong minimum principle holds for $\ut$: 
if $\Om$ is connected and $\ut$ attains
its minimum in $\Om$ then it must be constant.
\end{thm}

A bit surprising, perhaps, is that the weak minimum principle,
which compares infima on sets and their boundaries,
does not follow, see Example~\ref{ex-slit-plane} below.

Note that the local assumptions on $\mu$ in 
Theorem~\ref{thm-semiloc-int-reg} are required only in $\Om$, but
the ambient space $X$ plays an implicit role in the definition
of quasi(super)\-mi\-ni\-miz\-ers through the range
of test functions $\phi \in \Ndisto(\Om) \subset\Np(X)$.
Under global assumptions, (H\"older) continuity of (quasi)minimizers 
has been deduced on metric spaces in  
Kinnunen--Shanmugalingam~\cite{KiSh01},
Kinnunen--Martio~\cite{KiMa02}, \cite{KiMa03}
and Bj\"orn--Marola~\cite{BMarola}. 
In \cite{KiMa02} and \cite{KiMa03}
completeness was assumed but not used for these results,
although it certainly influenced their formulation
of the definition of quasi\-(super)\-mini\-mizers
(using $\Gcpt$ in our notation).

In addition to having stronger assumptions on $X$, these papers
also use more restrictive definitions of $\Nploc(\Om)$ 
and/or a larger class of test functions than here.
The local space $\Nploc(\Om)$ in~\cite{KiMa02} and~\cite{KiMa03}
coincides with our $\Ncptloc(\Om)$
and their test functions belong to $\Ncpto(\Om)$
(which equals $\Np_0(\Om)$ because of the assumed completeness), 
while \cite{BMarola} uses $\Ndistloc(\Om)$ and $\Ndisto(\Om)$.
In \cite{KiSh01}, the test functions belong to $\Ncloso(\Om)$,
while the definition of $\Nploc(\Om)$ is through bounded
subsets and imposes integrability conditions also near the boundary
$\bdry\Om$, so that it coincides with $\Np(\Om)$ for bounded $\Om$.

The smallest test space $\Ncpto(\Om)$ is used in~\cite{KiMa02} 
and~\cite{KiMa03}, as well as in Holopainen--Shanmugalingam~\cite{HoSh}
(in locally compact spaces).
Such a definition guarantees that \p-harmonicity in each $\Om_j$ for an 
increasing exhaustion $\Om_1\subset\Om_2\subset\cdots$ implies
\p-harmonicity in $\Om=\bigcup_{j=1}^\infty\Om_j$.
This need not be true with the other classes of test functions, as seen 
in Example~\ref{ex-slit-plane} below.
At the same time, in general spaces there may be 
no nonempty open sets with compact closures.

\begin{proof}[Proof of Theorem~\ref{thm-semiloc-int-reg}]
For each $x \in \Om$ there are a ball $B_x$ and $\la_x$ 
such that $u\in\Np(B_x)$
and the doubling property and the \p-Poincar\'e inequality hold
within $2B_x \subset \Om$ with dilation constant $\la_x$.
As $X$ is Lindel\"of we can find countably many balls
$\{B_j\}_{j=1}^\infty$ such that $B_j=(50\la_{x_j})^{-1}B_{x_j}$
and $\Om\subset\bigcup_{j=1}^\infty B_j$.

To deduce continuity, we need to obtain suitable weak Harnack inequalities 
for $u$ within each ball $B_j$, in the sense that 
they hold with fixed
constants (depending on $j$) for each ball $B \subset B_j$.
The arguments for proving such weak Harnack inequalities 
in Kinnunen--Shanmugalingam~\cite{KiSh01}  and 
Kinnunen--Martio~\cite[Section~5]{KiMa03} 
are all local (and do not use 
completeness), so local assumptions are enough for them.
They do rely on a better $q$-Poincar\'e inequality for some $q<p$ but
it is only applied to $L^p$-integrable \p-weak upper gradients and thus
inequality~\eqref{eq-KZ} provided by Theorem~\ref{thm-KZ-local} is sufficient.

If $u$ is a quasiminimizer then
it is a standard procedure using these weak Harnack inequalities 
to deduce continuity for a representative of $u$ in $\Om$, see
\cite{KiSh01} for the details.
This even gives local H\"older continuity, but without uniform
control of the H\"older exponent, since it locally depends on the constants
within each $B_j$, and the $B_j$
in turn depend both on $\Om$ and $u$.

If $u$ is a quasisuperminimizer then also the Lebesgue
point result provided by Proposition~\ref{prop-Leb-pt} is needed.
The lsc-regularity for a representative of $u$ in $\Om$ then
follows as in \cite[Theorem~5.1]{KiMa03} or 
Bj\"orn--Bj\"orn--Parviainen~\cite[Theorem~6.2]{BBParv}.

Finally, because of the weak Harnack inequalities, the strong minimum
principle for $\ut$ follows as in the proof of \cite[Theorem~8.13]{BBbook}.
\end{proof}

\begin{example}   \label{ex-slit-plane}
Let $X=\R^2\setm ([-1,\infty)\times\{0\})$ be the slit plane,
i.e.\ $\R^2$ with a ray removed.
Equip $X$ with the Euclidean metric and the Lebesgue measure.
Note that $X$ is locally compact and $\mu$ is globally doubling
and supports a local $1$-Poincar\'e inequality.
Also let $\Om=(-1,1)\times(0,2)\subset X$.

Since $\dist(G,X\setm\Om)= \dist(G,\R^2\setm\Om)$ for every open $G\subset\Om$,
it is easily seen that $\Gcpt(\Om)=\Gdist(\Om)$.
On the other hand, 
\[
G=\bigl( -\tfrac12,\tfrac12 \bigr) \times (0,1) \in \Gbdy(\Om)\setm\Gdist(\Om)
\]
shows that $\Gdist(\Om) \varsubsetneqq \Gbdy(\Om)$.
Similarly, the closure of 
\[
H=\{(x_1,x_2)\in\Om: |x_1|+x_2<1\}, 
\]
taken with respect to $X$, satisfies 
$\itoverline{H}\subset\Om$ and hence $H\in \Gclos(\Om)\setm\Gbdy(\Om)$.
We thus conclude that
\begin{equation} \label{eq-G-ex}
\Gcpt(\Om) = \Gdist(\Om) \varsubsetneqq \Gbdy(\Om) \varsubsetneqq \Gclos(\Om),
\end{equation}
which immediately implies that
\[
\Ndistloc(\Om) = \Np\loc(\Om) = \Ncptloc(\Om).
\]
On the other hand, the functions $|x-(1,1)|^{-1}$, $|x-(1,0)|^{-1}$ and $|x|^{-1}$
show that
\[
   \Np(\Om) \subsetneq \Nclosloc(\Om) \subsetneq \Nbdyloc(\Om) 
    \subsetneq \Ndistloc(\Om).
\]

For the zero spaces, it follows from \eqref{eq-G-ex} that
\[
\Npc(\Om)= \Ncpto(\Om) = \Ndisto(\Om).
\]
At the opposite end of the chain~\eqref{eq-incl-Np0} of zero spaces, 
it follows from \cite[Theorem~9.1]{BBsemilocal} 
that all functions in $\Np(X)$ are quasicontinuous and hence 
\cite[Lemma~5.43]{BBbook} implies that 
\[
\Ncloso(\Om)=\Np_0(\Om).
\]
Furthermore, by regarding $\Om$ as a subset of $\R^2$, we can conclude that
every function in $\Ndisto(\Om)$ extends by zero to a function in $\Np(\R^2)$
and hence has boundary values~0 q.e.\ on  $[-1,1]\times\{0\}$. 
Since it is easily verified that 
$\dist(\,\cdot\,,\bdry\Om)\in \Nbdyo(\Om)$, this implies that
\[
\Ndisto(\Om) \varsubsetneqq \Nbdyo(\Om).
\]
To investigate the remaining inclusion $\Nbdyo(\Om) \subset \Ncloso(\Om)$,
assume that $\phi\in\Np_0(G)$ for some $G\in\Gclos(\Om)$
with $\dist(G,\bdry\Om)=0$.
The only limit points (with respect to $\R^2$) that $G$ and $\bdry\Om$
can share, are $z_{\scriptscriptstyle\pm}=(\pm1,0)$.

Next, we distinguish between $p\le2$ and $p>2$.
Since singletons in $\R^2$ have zero \p-capacity, when $p\le2$, 
there exist Lipschitz cut-off functions 
$\eta_j$ supported in $B(z_\limplus,2/j)\cup B(z_\limminus,2/j)$ such that
\[
\eta_j=1 \text{ in } B(z_\limplus,1/j)\cup B(z_\limminus,1/j)
\quad \text{and} \quad 
\|\eta_j\|_{\Np(\R^2)}\to0 \text{ as } j\to\infty.
\]
The functions $(1-\eta_j)\phi$ then belong to $\Nbdyo(\Om)$ and
approximate any bounded $\phi$ in the $\Np(X)$-norm.
As unbounded functions can be approximated by their truncations,
this shows that 
\[
\Nbdyo(\Om)=\Ncloso(\Om) \quad \text{for } p\le2.
\]
For $p>2$, we proceed as follows.
Since $\mu$ is globally doubling and supports a global
1-Poincar\'e inequality
on the upper half-plane, Theorem~\ref{thm-Xhat-semi} implies that
$\phi$ extends to $\phihat\in\Np(\R\times[0,\infty))$, with the same
norm and minimal \p-weak upper gradient.
Moreover, $\phihat$ is continuous and $\phihat(z_{\scriptscriptstyle\pm})=0$.
This implies that the functions $(\phi-1/j)_\limplus\in\Nbdyo(\Om)$
approximate $\phi_{\limplus}$ in the $\Np(X)$-norm and thus
\[
\Nbdyo(\Om)=\Ncloso(\Om) \quad \text{also for } p>2.
\]

By varying both the local Newtonian space for (quasi)minimizers and the
class of test functions in~\eqref{eq-deff-qmin} one obtains different
definitions.
Let us have a closer look at some extreme cases:

1.\ Assume that $u\in\Np\loc(\Om)$ and test with 
$\phi\in\Ncpto(\Om)=\Ndisto(\Om)$.
This gives the most general definition and the largest class of 
(quasi)minimizers.
Since $\Om$ is open in $\R^2$, we have $\Np\loc(\Om)=W^{1,p}\loc(\Om)$,
the usual local Sobolev space on Euclidean domains.
It follows that this definition provides us with the usual \p-harmonic
functions and quasiminimizers on the Euclidean domain $\Om\subset\R^2$.

However, since the boundary $\bdry\Om$ with respect to $X$ does not
include the segment $[-1,1]\times\{0\}$, it is easily verified that
uniqueness is lost in the Dirichlet problem 
for \p-harmonic functions
when the boundary data
are only prescribed on $\bdry\Om$, e.g.\ by requiring that $u-f\in\Np_0(\Om)$.
A remedy of this problem is achieved by a larger class of test functions below.

Furthermore, the weak maximum principle is violated,
as well as certain (weak) Harnack inequalities with respect to balls
in $\Om\subset X$.
To see this, consider e.g.\ the usual fundamental solution 
\begin{equation}  \label{eq-fund-sol}
u(x)= \begin{cases}  |x|^{(p-2)/(p-1)}, & p\ne2, \\
                       - \log |x|, &p=2,
        \end{cases}
\end{equation}
for the \p-Laplacian $\Delta_pu :=\dvg(|\grad u|^{p-2}\grad u)$.

This suggests that testing only with $\phi\in \Ncpto(\Om)=\Ndisto(\Om)$
and allowing (quasi)minimizers to belong to $\Np\loc(\Om)$ 
may be too generous and that $\Nbdyloc(\Om)$ or $\Nclosloc(\Om)$,
together with larger classes of test functions,
might be better choices.
On the other hand, the strong maximum principle, stating that a nonconstant
\p-harmonic function cannot attain its maximum in $\Om$,
as well as (weak) Harnack inequalities with respect to compact subsets 
of $\Om$ remain true even in this situation.

2.\ The space $\Nbdyo(\Om)=\Ncloso(\Om)=\Np_0(\Om)$
allows for test functions which need not vanish on the real axis.
This indirectly forces the (quasi)minimizers on 
$\Om\subset X$ to have zero Neumann boundary values 
on $ (-1,1) \times\{0\} $,
i.e.\ on the ``missing'' boundary segment.
This fails e.g.\ for the linear function $(x_1,x_2)\mapsto x_2$, which 
is thus not \p-harmonic with this more restrictive definition.

As mentioned above, the zero Neumann condition on the ``missing'' boundary
may restore uniqueness in the Dirichlet problem (with respect to $X$).
Note, however, that for $p\le2$ the fundamental solution~\eqref{eq-fund-sol} 
still satisfies \eqref{eq-deff-qmin} with $Q=1$ 
for all test functions $\phi \in \Nbdyo(\Om)=\Ncloso(\Om)=\Np_0(\Om)$.
(Indeed, as singletons have zero \p-capacity, 
test functions in $\Nbdyo(\Om)$
can be approximated therein by test functions vanishing near the singularity 
$(0,0)$.)

A rather general existence and uniqueness result for the Dirichlet
problem for \p-harmonic functions with Sobolev boundary values
was given by Bj\"orn--Bj\"orn~\cite[Theorem~4.2]{BBnonopen},
which covers the case considered here.
There the functions are required to belong to $D^p(\Om)$ and the test
function space is $\Np_0(\Om)$.

3.\ For $p\le2$, the fundamental solution $u$ in~\eqref{eq-fund-sol} 
would be excluded (and the uniqueness in the Dirichlet problem restored)
if \p-harmonic functions were required to
belong to $\Nclosloc(\Om)$ or $\Nbdyloc(\Om)$.
However, the translated fundamental solutions satisfy 
\[
u(x-(0,1)) \in \Nbdyloc(\Om) \setm \Nclosloc(\Om)
\quad \text{and} \quad 
u(x-(1,1)) \in \Nclosloc(\Om) \setm \Np(\Om)
\]
and both can be tested by $\phi\in\Np_0(\Om)$.

For $p>2$, the fundamental solution~\eqref{eq-fund-sol} belongs to 
$\Np(\Om)$, but testing~\eqref{eq-deff-qmin} with 
\[
\phi=(1-u)_\limplus \in \Nbdyo(\Om) = \Ncloso(\Om)=\Np_0(\Om)
\]
shows that it is not \p-harmonic in $\Om\subset X$ with such a definition.
It is, however, a subminimizer with this class of test functions.

The above observations concerning the fundamental solutions hold also 
for the power functions $x\mapsto|x|^\al$
with $\al<1-2/p<0$ and $\al>1-2/p>0$, 
which are quasiminimizers in $\R^2\setm\{0\}$ for $p<2$ and $p>2$,
respectively, in view of 
Bj\"orn--Bj\"orn~\cite[Theorems~5.1 and~6.1]{BBpower}.
\end{example}

We have thus seen that the different possible definitions have various
pros and cons, and that the ``correct'' definition depends on
the particular applications or results one has in mind.
For example, suitable choices of spaces of test functions 
in noncomplete spaces also
make it possible to treat certain mixed boundary value problems within the 
scope of Dirichlet problems.

A seemingly simple way of treating (quasi)minimizers on noncomplete spaces
might be to use the completion $\Xhat$ of $X$
together with our main extension theorem (Theorem~\ref{thm-Xhat-semi}):
Starting with a quasiminimizer $u$ on some open subset $\Om$
of $X$ one would like to extend (a representative of) $u$
to a function $\uhat$ on some open subset $\Ghat$ of $\Xhat$,
which can be achieved using Lemma~\ref{lem-ext-Nploc}.
The next step would be to show that $\uhat$ is a quasiminimizer
in $\Ghat$, and then apply the potential theory for
quasiminimizers on $\Xhat$.

There are several conditions that need to be fulfilled for such 
an approach to be fruitful. 
First of all, in order to have a useful potential theory
on $\Xhat$, we need to assume that $\mu$ is locally
doubling  and supports a local \p-Poincar\'e inequality
on $\Xhat$. 
In view of Corollaries~\ref{cor-semilocal-doubling}
and~\ref{cor-Xhat-PI},
it seems that the most natural condition to impose on $X$ to achieve this 
is requiring that $\mu$ is semilocally doubling and
supporting a semilocal \p-Poincar\'e inequality on $X$,
which
ensures that 
these semilocal conditions also hold on $\Xhat$.
By \cite[Theorem~4.4 and the discussion following it]{BBsemilocal},
it follows that $\Xhat$ is also proper and connected.
Hence, most of the nonlinear potential
theory on metric spaces is available for $\Xhat$, see \cite[Section~10]{BBsemilocal}.

So assume that $\mu$ is semilocally doubling 
and supports a semilocal \p-Poincar\'e inequality on $X$.
The proof of Lemma~\ref{lem-ext-Nploc} shows that 
(a representative of) any
$u\in\Nxloc(\Om)$
extends to a function $\uhat \in \Nxloc(\Omhat)$ 
if ${\bf x} \in \{\text{dist}, \text{bdy},\text{clos}\}$.
For $u \in\Np\loc(\Om)$ there is only an extension to some 
open $\Ghat\subset\Xhat$ which may depend on $u$, see Example~\ref{ex-G-dep-u}.

Now, if~\eqref{eq-Xhat-local} in Theorem~\ref{thm-Xhat-semi} holds
for all such extensions with a uniform $A_0\ge1$
(e.g.\ if $X$ is locally compact or  
\p-path open in 
$\Xhat$, in which case $A_0\equiv1$) then~\eqref{eq-deff-qmin} implies that also
\[
      \int_{\phi \ne 0} g^p_{\uhat} \, d\mu 
            \le Q A_0^p \int_{\phi \ne 0} g_{\uhat+\phi}^p \, d\mu
\]
for all $\phi \in \Nyo(\Ghat)$
since $\phi|_X \in \Nyo(\Om)$.
In other words, $\uhat$ is a 
$QA_0^p$-quasi\-\/\textup{(}super\/\textup{)}\-mini\-mizer 
in $\Ghat$, where $\Ghat=\Omhat$
if ${\bf x} \in \{\text{dist}, \text{bdy},\text{clos}\}$.
(As before we omit the cases when ${\bf x} = {\rm cpt}$ or 
${\bf y} = {\rm cpt}$.)
This, in particular, implies that various local properties,
such as the H\"older continuity,
(weak) Harnack inequalities and maximum principles, hold for $\uhat$
in $\Ghat$, and thus also for $u$ in $\Om$.
When $A_0 \equiv 1$, even more can be said.

Another point is whether there is a one-to-one correspondence
between the (equivalence classes of) $Q$-quasiminimizers on
$\Om$ and on $\Omhat$, when $A_0 \equiv 1$.
From
\[
\Gdist(\Omhat)=\{\Ghat\subset\Omhat: \Ghat\cap X \in \Gdist(\Om)\}
\]
it follows that 
$\Ndistloc(\Omhat)=\Ndistloc(\Om)$
and $\Ndisto(\Omhat)=\Ndisto(\Om) $
(or more precisely there is a one-to-one correspondence between
the equivalence classes in these spaces),
which shows that such an equivalence holds
if ${\bf x}={\bf y}=\text{dist}$ (under the assumptions above).
On the other hand, 
this fails for the other families 
$\Gcpt$, $\Gbdy$ and $\Gclos$, and thus such
a correspondence is unlikely to hold in any other case.
Example~\ref{ex-slit-plane} gives
several counterexamples when applied
to $X_\limplus =X \cap (\R \times [0,\infty))$ on which global assumptions
hold.

Using the test function class $\Np_0(\Om)$, 
the Dirichlet and obstacle problems
on bounded (not necessarily open) sets in noncomplete spaces
with very weak assumptions were studied
in Bj\"orn--Bj\"orn~\cite{BBnonopen}.
Functions considered therein belong to $\Np(\Om)$,
so the different types of local spaces do not play a role in that
discussion.
The space $\Np_0(\Om)$ of test functions therein is large enough
to give a sufficiently restrictive definition of \p-harmonic functions
which, under rather mild assumptions, guarantees uniqueness in the
Dirichlet problem, cf.\ Parts~1--3 in our Example~\ref{ex-slit-plane}.

It was also shown in~\cite{BBnonopen}
that in complete spaces (with global assumptions)
Dirichlet and obstacle problems are naturally studied on 
quasiopen (or finely open) sets, but not really beyond that.
In our setting it could therefore be interesting to know what happens
when $X$ (and thus $\Om$) is quasiopen in $\Xhat$,
which
is closely related to its \p-path openness, 
see the comments after Theorem~\ref{thm-Xhat-semi}.

For a fruitful nonlinear potential theory in noncomplete spaces it may
be worth to consider how the fine potential theory on quasiopen (and finely open)
sets has been developed in $\R^n$ and in metric spaces, 
and in particular the role of so-called \p-strict
subsets, 
see Kilpel\"ainen--Mal\'y~\cite{KiMa92}, Latvala~\cite{LatPhD} and
 Bj\"orn--Bj\"orn--Latvala~\cite{BBLat3}.

In connection with the Dirichlet problem it would be interesting 
to develop a suitable theory for Perron solutions 
in noncomplete spaces. 
A major obstacle may however be the comparison principle
(as in \cite[Theorem~9.39]{BBbook}), since even
in the complete case (and with global assumptions)
it is not known whether
the boundary condition (for bounded functions) can be omitted
even at a single
point with zero capacity.
Perhaps a suitable theory could be developed if one assumes that the boundary is
compact, even though the underlying space may be noncomplete.

On the other hand, it would be interesting to study how continuous boundary
data should be treated on noncompact boundaries, in which case
there are at least three natural counterparts to the usual space
of continuous boundary data: continuous functions, bounded continuous
functions and uniformly continuous functions.
See Bj\"orn~\cite{ABcomb} for one study, in a very special case, treating
Perron solutions for \p-harmonic functions with a noncompact boundary;
and also Estep--Shanmugalingam~\cite{ES}.

\end{document}